\documentclass[11pt, reqno]{amsart}
\usepackage{amsmath,amsfonts,amssymb,amsthm}
\usepackage{epsfig}

\usepackage{amstext}
\usepackage{xspace}
\usepackage{amsfonts}
\usepackage{amsmath}
\usepackage{amssymb}
\usepackage{amstext}
\usepackage{amsthm}       
\usepackage{xspace}
\usepackage[active]{srcltx}

\newcommand{\DD}{\mathcal D}
\newcommand{\R}{\mathbb R}

\newcommand{\N}{\mathbb N}

\newcommand{\T}{\mathbb T}
\newcommand{\Z}{\mathbb Z}

\newcommand{\sol}{\mathfrak{S}}

\newcommand{\comp}{\mbox{\scriptsize  $\circ$}}
\newcommand{\eps}{\varepsilon}
\newcommand{\ucv}{\rightrightarrows}

\newcommand{\A}{\mathcal{A}}
\newcommand{\hh}{\mathcal{H}}

\newcommand{\E}{\mathcal{E}}
\newcommand{\K}{\mathcal{K}}

\renewcommand{\S}{\mathcal{S}}

\newcommand{\tagliato}{$\kern-5 mm -$}
\newcommand{\tagliat}{$\kern-4 mm -$}

\newcommand{\tagli}{\mbox{\tagliat}}

\newcommand{\abra}[1]{(\ref{#1})}

\newcommand{\D}[1]{\mbox{\rm #1}}
\newcommand{\dd}{\D{d}}

\newcommand{\weakcv}{\rightharpoonup}

\newtheorem{teorema}{Theorem}[section]
\newtheorem{prop}[teorema]{Proposition}
\newtheorem{lemma}[teorema]{Lemma}
\newtheorem{definition}[teorema]{Definition}

\newtheorem{guess}[teorema]{Remark}
\newtheorem{example}[teorema]{Example}

\newenvironment{dimo}{{\bf\noindent Proof.}}{\qed}
\newenvironment{oss}{\begin{guess} \begin{rm}}{\end{rm} \end{guess}}

\begin{document}

\title{Weak KAM theoretic aspects \\for nonregular commuting Hamiltonians}
\author{Andrea Davini \and Maxime Zavidovique}
\address{Dip. di Matematica, Universit\`a di Roma ``La Sapienza",
P.le Aldo Moro 2, 00185 Roma, Italy}
\email{davini@mat.uniroma1.it}
\address{
IMJ, Universit\' e Pierre et Marie Curie, Case 247, 
4 place Jussieu, 
F-75252 PARIS CEDEX 05, FRANCE}
\email{zavidovique@math.jussieu.fr} 

\begin{abstract}
In this paper we consider the notion of commutation for a pair of
continuous and convex Hamiltonians, given in terms of commutation
of their Lax--Oleinik semigroups. This is equivalent to the
solvability of an associated multi--time Hamilton--Jacobi
equation. We examine the weak KAM theoretic aspects of the
commutation property and show that the two Hamiltonians have the
same weak KAM solutions and the same Aubry set, thus generalizing
a result recently obtained by the second author for Tonelli
Hamiltonians. We make a further step by proving that the
Hamiltonians admit a common critical subsolution, strict outside
their Aubry set. This subsolution can be taken of class $C^{1,1}$
in the Tonelli case. To prove our main results in full generality,
it is crucial to establish suitable differentiability properties
of the critical subsolutions on the Aubry set. These latter
results are new in the purely continuous case and  of
independent interest.
\end{abstract}
\maketitle
\begin{section}{Introduction}
In the last decades, the study  of Hamiltonian systems has been
impacted by a few new tools and methods. For general Hamiltonians,
the framework of symplectic geometry led Gromov to his
non--squeezing lemma \cite{Gromov}, which gave rise to the key
notion of symplectic capacity, now uniformly used in the field.

In the particular case of Tonelli (smooth, strictly convex,
superlinear) Hamiltonians,  some variational techniques led to
significant improvements and results. John Mather led the way in
this direction in \cite{MaZ,Mat}. In the first paper he studied
free time minimizers of the Lagrangian action functional,
introducing the Aubry set, while in the second one he studied
invariant minimizing measures, introducing what is now called the
Mather set.

Later Fathi, through his weak KAM Theorem and Theory, showed the
link between the variational sets introduced by Mather and the
Hamilton--Jacobi equation. This allowed to simplify some proofs of
Mather and to establish new PDE results, in connection with the
theory of homogenization \cite{LPV}. This material is presented in
\cite{Fathi}.

The main challenge now seems to find analogues of the
Aubry--Mather theory in wider settings. There are mainly two
approaches to this problem. The first one is to lower the
regularity of the Hamiltonians. This is a rather natural issue in
view of applicability to Optimal Control and Hamilton--Jacobi
equations. A generalization of the weak KAM theory to continuous
and quasi--convex Hamiltonians was first given by Fathi and
Siconolfi in \cite{FaSic03}. Their approach has been subsequently
developed and applied in different contexts, see for instance
\cite{CaCeSic09, DavSic05, DavSic09, DavSic11, DavSic12, FatSic11,
IcIs08, IcIs09, Ishii, IsMi, SicTer}.

The second is to drop the convexity and coercivity assumptions,
thus preventing from using traditional variational arguments. A
generalization of weak KAM Theory to this framework is an
outstanding and widely open question. On the other hand, the
theory of viscosity solutions, introduced by Crandall and Lions
\cite{CrLi}, provides powerful tools to study Hamilton--Jacobi
equations in broad generality. With regard to the problems studied
in the references above mentioned, these techniques have been
successfully employed to obtain similar results under different,
and in some cases weaker, assumptions on the Hamiltonians, see for
instance \cite{AgAn, AlBa, Ba07, BaMi, BaSo00, Carda}.\smallskip

The present paper is addressed to explore the weak KAM theoretic
aspects of commuting Hamiltonians. This issue is related to the
solvability of a multi--time Hamilton--Jacobi equation of the kind
\begin{equation}\label{multi-time HJ}
\begin{cases}
\displaystyle
{\partial_t u}+H(x, D_x u)=0&\quad \hbox{in $(0,+\infty)\times (0,+\infty)\times M$}\medskip\\
\displaystyle {\partial_s u}+G(x, D_x u)=0&\quad \hbox{in $(0,+\infty)\times (0,+\infty)\times M$}\medskip\\
u(0,0,x)=u_0(x)&\quad \hbox{on $M$},
\end{cases}
\end{equation}
where $M$ stands either for the Euclidean space $\R^N$ or the
$N$--dimensional flat torus $\T^N$, $H$ and $G$ denote two real
valued functions on $M\times \mathbb{R}^N$, and $u_0:M \to \R$ is
any given Lipschitz continuous initial datum.

The first existence and uniqueness results appeared in \cite{LR}
for Tonelli Hamiltonians independent of $x$ via a representation
formula for solutions of the Hamilton--Jacobi equation: the
Hopf--Lax formula. Related problems were studied in \cite{imbert}.

A generalization of this result came much later in \cite{BT},
where dependance in $x$ is introduced (and the convexity hypothesis is
kept). As a counterpart, the authors explain the necessity to
impose the following commutation property on the Hamiltonians:
\begin{equation}\label{commutation}
    \langle D_x G,\, D_p H\rangle -\langle D_x H,\, D_p G\rangle =0\qquad\hbox{in $M\times\R^N$}.
\end{equation}
Note that this condition is automatically  satisfied when the
Hamiltonians are independent of $x$. The proof involves an a
priori different Hamilton--Jacobi equation with parameters and
makes use of fine viscosity solution techniques. In \cite{Za},
under stronger hypotheses, a more geometrical proof, following the
original idea of Lions--Rochet, is given.

This equation was then studied under weaker regularity assumptions
in \cite{MR}. The convexity is dropped in \cite{CaVi08} in the
framework of symplectic geometry and variational solutions.
Finally, let us mention that in \cite{So} the influence of first
integrals (not necessarily of Tonelli type) on the dynamics of a
Tonelli Hamiltonian is studied.

In \cite{Za} the second author has explored relation
\eqref{commutation} for a pair of Tonelli Hamiltonians in the
framework of weak KAM Theory to discover that the notions of Aubry
set, of Peierls barrier and of weak KAM solution are invariants of
the commutation property. Similar results were independently obtained in \cite{chinois,chinois1}. \smallskip

This article deals with the first approach: we will consider
purely continuous Hamiltonians, but we will keep the convexity and
coercivity assumptions in the gradient variable. Here, we will say that $H$ and $G$ commute to  
simply mean that the multi--time Hamilton--Jacobi equation
\eqref{multi-time HJ} admits a viscosity solution for any
Lipschitz initial datum. This is formulated in terms of
commutation of their Lax--Oleinik semigroups, and is equivalent to
\eqref{commutation} when the Hamiltonians are smooth enough, see
\cite{BT} and Appendix \ref{appendix commutation}. 

The purpose of the paper is 
to explore the weak KAM consequences of the commutation property in this setting. Our main 
achievement in this direction is the following Theorem, which generalizes the main new result of \cite{Za}:
\begin{teorema}\label{teo main}
Let $H$ and $G$ be a pair of continuous, strictly convex and
superlinear Hamiltonians on $\T^N\times\R^N$. If $H$ and $G$
commute, then they have the same weak KAM (or critical) solutions
and the same Aubry set.
\end{teorema}

We want to emphasize that the extension of this result to the non--regular setting was far from being straightforward. First, there is a problem of techniques: the crucial point in the proof of \cite{Za} is based on a careful study of the flows associated with $H$ and $G$ and exploits properties and tools
developed in the framework of symplectic geometry and weak KAM
Theory. For instance, a key ingredient is a deep result due to Bernard \cite{BeSy}, stating that the Aubry set is a symplectic invariant.   
The use of all this machinery is only possible under strong
regularity assumptions on the Hamiltonians. 

But there is more: looking at the arguments in \cite{Za}, one realizes that  
the commutation hypothesis \eqref{commutation} entails a
certain rigidity of the dynamics and of the underlying geometric
frame of the equations. Even if in the purely
continuous case some analogies can be drawn, all this rich
structure is lost. To put it differently, the problem did not seem to be just 
of technical nature: there was no {\em a priori} evidence that the
result would continue to hold by dropping the regularity.

The proof given here borrows some arguments from \cite{Za}, but the
conclusion is reached via a different and rather simple remark   
on the time--dependent equations. Incidentally, with this idea the proof in the smooth case can be 
made considerably simpler.  It is also worth noticing that Theorem \ref{teo main} applies, in particular, to  
a pair of Hamiltonians of class $C^1$ satisfying \eqref{commutation}, that is, to a case not covered 
by the previous works on the subject \cite{chinois, chinois1, Za}. As a byproduct, our study allows us to obtain a new result also for
classical Tonelli Hamiltonians:

\begin{teorema}
Let $G$ and $H$ be two commuting  Tonelli Hamiltonians. Then they
admit a $C^{1,1}$ critical subsolution which is strict outside
their common Aubry set.
\end{teorema}

In the end, our research reveals that the invariants observed in
the framework of weak KAM Theory are consequence of the
commutation of the Lax--Oleinik semigroups only, with no further
reference to the Hamiltonian {flows},  that cannot be even defined
in our setting. The only point where a kind of generalized
dynamics plays a role is when we establish some differentiability
properties of critical subsolutions on the Aubry set, which are
crucial to state Theorem \ref{teo main} in its full generality.
These results are presented in Section \ref{sez
differentiability}, where we will prove a more precise version of
the following

\begin{teorema} Let $H$ be a continuous, strictly convex and superlinear Hamiltonian on $\T^N\times\R^N$.
Then there exists a set $\DD \subset \T^N$ such that any
subsolution $u$ of the critical Hamilton--Jacobi equation is
differentiable on $\DD$. Moreover, its gradient $D u$ is
independent of $u$ on $\DD$. Last, $\DD$ is a uniqueness set for
the critical equation, that is, if two weak KAM (or critical)
solutions coincide on $\DD$, then they are in fact equal.
\end{teorema}

These latter results are new and we believe interesting {\em per
se}. They generalize, in a weaker form, Theorem 7.8 in
\cite{FaSic03}, and bring the hope of extending to the purely
continuous case the results of \cite{FaSic03} about the existence
of a $C^1$ critical subsolution, strict outside the Aubry set.
Such a generalization, however, seems out of reach without any
further idea.\medskip

The article is organized as follows. In Section \ref{sez notation}
we present the main notations and assumptions used throughout the
paper, while in Section \ref{sez HJ} we recall the definitions and
the results about Hamilton--Jacobi equations that will be needed
in the sequel. Section \ref{sez weak KAM} consists in a brief
overview of weak KAM Theory for non--regular Hamiltonians. Some
proofs are postponed to Appendix \ref{appendix b}. In Section
\ref{sez differentiability} we prove the differentiability
properties of critical subsolutions above mentioned. In Section
\ref{sez commuting} we examine the weak KAM
theoretic aspects of the commutation property and we establish our
main results for continuous and strictly convex Hamiltonians. Some auxiliary lemmas 
are stated and proved in Appendix \ref{appendix technical lemmas}. Appendix \ref{appendix commutation} 
contains an argument showing the equivalence between the notion of commutation considered in this paper and the  
one given in terms of cancellation of the Poisson
bracket when the Hamiltonians are of Tonelli type.\\
\end{section}

\begin{section}{Preliminaries}

\begin{subsection}{Notations and standing assumptions}\label{sez notation}
With the symbols $\R_+$ and $\R_-$ we will refer to the set of
nonnegative and nonpositive real numbers, respectively.  We say
that a property holds {\em almost everywhere} ($a.e.$ for short)
on $\R^k$ if it holds up to a {\em negligible} subset, i.e. a
subset of zero $k$--dimensional Lebesgue measure.

\indent By modulus we mean a nondecreasing function from $\R_+$ to
$\R_+$, vanishing and continuous at $0$. A function $g:\R_+\to\R$
will be termed {\em superlinear} if
\[
\lim_{h\to +\infty} \frac{g(h)}{h}=+\infty.
\]

\noindent Given a metric space $X$, we will write
$\varphi_n\ucv\varphi$ on $X$ to mean that the sequence of
functions $(\varphi_n)_n$ uniformly converges to $\varphi$ on
compact subsets of $X$. Furthermore, we will denote by
$\D{Lip}(X)$ the family of Lipschitz--continuous real functions
defined on $X$.\smallskip\par

Throughout the paper, $M$ will refer either to the Euclidean space
$\R^N$ or to the $N$--dimensional flat torus $\T^N$, where $N$ is
an integer number. The scalar product in $\R^N$ will be denoted by
$\langle\,\cdot\;, \cdot\,\rangle$, while the symbol $|\cdot|$
stands for the Euclidean norm. Note that the latter induces a norm
on $\T^N$, still denoted by $|\cdot|$, defined as
\[
|x|:=\min_{\kappa\in\Z^N}|x+k|\qquad\hbox{for every $x\in\T^N$}.
\]
We will denote by $B_R(x_0)$ and $B_R$ the closed balls in $M$ of
radius $R$ centered at $x_0$ and $0$, respectively.

With the term {\em curve}, without any further specification, we
refer to an absolutely continuous function from some given
interval $[a,b]$ to $M$. The space of all such curves is denoted
by $W^{1,1}([a,b]; M)$, while $\D{Lip}_{x,y}([a,b]; M)$ stands for
the family of Lipschitz--continuous curves $\gamma$ joining $x$ to
$y$, i.e. such that $\gamma (a)=x$ and $\gamma (b)=y$, for any
fixed $x$, $y$ in $M$.
\par

With the notation $\|g\|_\infty$ we will refer to the usual
$L^\infty$--norm of $g$, where the latter will be either a
measurable real function on $M$ or a vector--valued measurable map
defined on some interval.
\par

Let $u$ be a continuous function on $M$. A  {\em subtangent} (respectively, {\em supertangent}) of $u$ at $x_0$ is a function $\phi\in C^1(M)$ such that $\phi(x_0)= u(x_0)$ and $\phi(x)\leqslant u(x)$  for every $x\in M$ (resp., $\geqslant$). Its gradient $D\phi(x_0)$ will be called a {\em subdifferential}  (resp. {\em superdifferential}) of $u$ at $x_0$, respectively. The set of sub and superdifferentials of $u$ at $x_0$ will be denoted $D^-u(x_0)$ and 
$D^+u(x_0)$, respectively. We recall that $u$ is
differentiable at $x_0$ if and only if $D^+u(x_0)$ and $D^-u(x_0)$ are
both nonempty. In this instance, $D^+u(x_0)=D^-u(x_0)=\{Du(x_0)\}$,
where $Du(x_0)$ denotes the differential of $u$ at $x_0$. We refer the
reader to \cite{CaSi00} for the 
proofs.

When $u$ is locally Lipschitz in $M$, we will denote by $\partial^*u(x_0)$ the set of 
{\em reachable gradients} of $u$ at $x_0$, that is the set
\[
\partial^* u(x_0)=\{\lim_n D u(x_n)\,:\,\hbox{$u$ is differentiable at $x_n$, $x_n\to x_0$}\,\},
\]
while the {\em Clarke's generalized gradient} $\partial_c u(x_0)$ is the closed convex hull of $\partial^* u(x_0)$. 
The set $\partial_c u(x_0)$ contains both $D^+u(x_0)$ and $D^-u(x_0)$, in particular $Du(x_0)\in \partial_c u(x_0)$ at any differentiability point $x_0$ of $u$.  
We recall that the set--valued map $x\mapsto\partial_c u(x)$ is upper semicontinuous with respect to set inclusion. When $\partial_c u(x_0)$ reduces to a singleton, the function $u$ is said to be {\em strictly differentiable} at that point. In this instance, $u$ is differentiable at $x_0$ and its gradient is continuous at $x_0$. We refer the reader to \cite{Cl} for a detailed treatment of the subject.

A function $u$ will be said to be {\em semiconcave} on an open subset $U$ of $M$ if for every $x\in U$ there exists  a vector $p_x\in\R^N$ such that 
\begin{equation*}\label{cond semiconcavita}
 u(y)-u(x)\leqslant \langle p_x,y-x\rangle + |y-x|\,\omega(|y-x|)\qquad\hbox{for every $y\in U$,}
\end{equation*}
where $\omega$ is a modulus. The vectors $p_x$ satisfying such inequality are precisely the elements of $D^+ u(x)$, which is thus always nonempty in $U$. Moreover, $\partial_c u(x)=D^+u(x)$ for every $x\in U$, yielding in particular that $Du$ is continuous on its domain of definition in $U$, see \cite{CaSi00}. This property will be exploited in the proof of Lemma \ref{lemma pain}.

\smallskip\par

Throughout the paper, we will call {\em convex Hamiltonian} a
function $H$ satisfying the following set of assuptions:\smallskip
\begin{itemize}
    \item[(H1)] \quad $H:M\times\R^N\to\R\qquad$ is
    continuous;\smallskip\\
    \item[(H2)] \quad $p\mapsto H(x,p)\qquad\hbox{is  convex on $\R^N$ for any
    $x\in M$;}$ \smallskip\\
     \item[(H3)] \quad there
     exist two  superlinear functions $\alpha,\beta:\R_+\to\R$ such
     that
     \[
     \alpha\left(|p|\right)\leqslant H(x,p)\leqslant \beta\left(|p|\right)\qquad\hbox{for all
     $(x,p)\in M\times\R^N$.}
     \]
\end{itemize}
\ \\
We  define the {\em Fenchel transform} $L:M\times\R^N\to \R$
of $H$ via
\begin{equation}\label{def L}
L(x,q)=H^*(x,q):=\sup_{p\in\R^N}\left\{\langle p, q\rangle -
H(x,p)\right\}.
\end{equation}
The function $L$ is called the Lagrangian associated with the
Hamiltonian $H$; it satisfies the following properties, see Appendix A.2 in \cite{CaSi00}:\\
\begin{itemize}
    \item[(L1)] \quad $L:M\times\R^N\to\R\qquad$ is
    continuous;\smallskip\\
    \item[(L2)] \quad $q\mapsto L(x,q)\qquad\hbox{is convex on $\R^N$ for any
    $x\in M$;}$ \smallskip\\
     \item[(L3)] \quad there
     exist two superlinear functions $\alpha_*,\beta_*:\R_+\to\R$
     s.t.
     \[
     \alpha_*\left(|q|\right)\leqslant L(x,q)\leqslant \beta_*\left(|q|\right)\qquad\hbox{for all
     $(x,q)\in M\times\R^N$.}\medskip
     \]
\end{itemize}
\begin{oss}\label{oss H3}
The functions $\alpha_*,\,\beta_*$ and $\alpha,\,\beta$ in (L3)
and (H3), respectively, can be taken continuous (in fact, convex),
without any loss of generality.
\end{oss}
\smallskip

With the term {\em strictly convex Hamiltonian} we will refer to a
convex Hamiltonian with (H2) replaced by the following stronger
assumption:\bigskip
\begin{itemize}
\item[(H2)$'$] \quad $p\mapsto H(x,p)\qquad\hbox{is strictly
convex on $\R^N$ for any $x\in M$}$.
\end{itemize}
\medskip
We point out that, in this event, $L$ enjoys\medskip
\begin{itemize}
\item[(L2)$'$] \quad $q\mapsto L(x,q)\qquad\hbox{is convex and of
class $C^1$ on $\R^N$ for any
    $x\in M$.}$\medskip
\end{itemize}
Furthermore, the map $(x,q)\mapsto D_q L(x,q)$ is continuous in
$M\times\R^N$. This fact will be exploited in the proof of
Proposition \ref{prop Fathi+}. Here and in the sequel, $D_q
L(x,q)$ and $D_x L(x,q)$ denote the partial derivative of $L$ at
$(x,q)$ with respect to $q$ and $x$, respectively. An analogous
notation will be used for the Hamiltonian.\smallskip\par

A {\em Tonelli Hamiltonian} is a particular kind of Hamiltonian
satisfying conditions (H1), (H2)$'$ and (H3). It is additionally
assumed of class $C^2$ in $M\times\R^N$ and condition (H2)$'$ is
strengthen by requiring, for every $(x,p)\in M\times\R^N$, that
\begin{equation}\label{eq Tonelli}
\frac{\partial^2 H}{\partial p^2}\,(x,p)\quad\hbox{is positive
definite as a quadratic form.}
\end{equation}
The associated Lagrangian has the same regularity as $H$ and
enjoys the analogous condition \eqref{eq Tonelli}.

 \vspace{1ex}
\end{subsection}

\begin{subsection}{Hamilton--Jacobi equations}\label{sez HJ}
Let us consider a family of Hamilton--Jacobi equations of the kind
\begin{equation}\label{eq hja}
H(x,Du)=a\qquad\hbox{in $M$,}
\end{equation}
where $a\in\R$. In the sequel, with the term {\em subsolution}
(resp. {\em supersolution}) of \eqref{eq hja} we will always refer
to a continuous function $u$ which is a subsolution (resp. a
supersolution) of \eqref{eq hja} in the viscosity sense, i.e. for
every $x\in M$
\begin{eqnarray*}
\qquad &H(x,p)\leqslant a& \qquad\hbox{for any $p\in D^+ u(x)$}\\
\big(\hbox{resp.}\quad &H(x,p)\geqslant a& \qquad \hbox{for any
$p\in D^- u(x)$}\ \big).
\end{eqnarray*}
A function will be
called a solution of \eqref{eq hja} if it is both a subsolution
and a supersolution.  
\smallskip
\begin{oss}\label{oss subsol}
Since $H$ is coercive, i.e. satisfies the first inequality in
(H3), it is well known that any continuous viscosity subsolution
$v$ of \eqref{eq hja} is Lipschitz, see for instance \cite{Ba94}.
In particular, $v$ is an almost everywhere subsolution, i.e.
\[
H\big(x,Dv(x)\big)\leqslant a \qquad\hbox{for a.e. $x\in M$}.
\]
By the convexity assumption (H2), the converse holds as well: any
Lipschitz,  almost everywhere subsolution solves \eqref{eq hja} in
the viscosity sense, see \cite{Sic09}. In particular, $v$ is a
subsolution of \eqref{eq hja} if and only if $-v$ is a subsolution
of
\[
H(x,-Du)=a\qquad\hbox{in $M$.}
\]
\end{oss}

 We define the {\em critical value} $c$ as
\[
c=\min\{a\in\R\,:\,\hbox{equation \eqref{eq hja} admits
subsolutions}\,\}.
\]
\smallskip
Following \cite{FaSic03}, we carry out the study of properties of
subsolutions of \abra{eq hja}, for $a \geqslant c$, by means of
the semidistances $S_a$ defined on $M\times M$ as follows:
    \begin{equation}\label{eq S}
    S_a(x,y)=\inf\left\{\int_0^1 \sigma_a\big(\gamma(s),\dot\gamma(s)\big)\,\dd
    s\,:\, \gamma\in\D{Lip}_{x,y}([0,1];M)\, \right\},
    \end{equation}
where $\sigma_a(x,q)$ is the support function of the $a$--sublevel
$Z_a(x)$ of $H$, namely
    \begin{equation}\label{sigma}
    \sigma_a(x,q):=\sup\left\{\langle q,p\rangle\,:\,p\in Z_a(x)
    \,\right\}
    \end{equation}
and $Z_a(x):=\{p\in\R^N\,:\,H(x,p)\leqslant a\,\}$. The function
$\sigma_a(x,q)$ is convex in $q$ and upper semicontinuous in $x$
(and even continuous at  points such that $Z_a(x)$ has nonempty
interior or reduces to a point), while $S_a$ satisfies the
following properties:
\begin{eqnarray*}
S_a(x,y)&\leqslant& S_a(x,z)+S_a(z,y)\\
S_a(x,y)&\leqslant& \kappa_a|x-y|
\end{eqnarray*}
for all $x,y,z\in M$ and for some positive constant $\kappa_a$.
The following properties hold, see \cite{FaSic03}:

\begin{prop}\label{prop S} Let $ a \geqslant c$.
\begin{itemize}
    \item[\em(i)] A function $\phi$ is a viscosity subsolution of \abra{eq hja} if and
    only if
\[
\phi(x)-\phi(y)\leqslant S_a(y,x)\qquad\hbox{for all $x,y\in M$.}
\]
In particular, all viscosity subsolutions of \eqref{eq hja} are
$\kappa_a$--Lipschitz continuous.\smallskip
    \item[\em(ii)] For any $y\in M$, the functions $S_a(y,\cdot)$
and
    $-S_a(\cdot,y)$ are both subsolutions of \abra{eq hja}.\smallskip
    \item[\em(iii)] For any $y\in M$
\[
S_a(y,x)=\sup\{v(x)\,:\,\hbox{$v$ is a subsolution to \eqref{eq
hja} with $v(y)=0$}\,\}.
\]
In particular, by maximality, $S_a(y,\cdot)$ is a viscosity
solution of \eqref{eq hja} in $M\setminus\{y\}$.\\
\end{itemize}
\end{prop}

\begin{definition}
For $t>0$ fixed, let us define the function $h^t:M\times M\to\R$
by
\begin{equation}\label{def h^t}
h^t(x,y)=\inf\left\{\int_{-t}^0 L(\gamma,\dot\gamma)\,\dd s\ :\
\gamma\in W^{1,1}([-t,0];M),\ \gamma(-t)=x,\,\gamma(0)=y \right\}.
\end{equation}
\end{definition}

It is well known, by classical result of Calculus of Variations,
that the infimum in \eqref{def h^t} is achieved. The curves that
realize the minimum are called {\em Lagrangian minimizers}. The
following more precise result will be needed in the sequel, see
\cite{AmAsBu89, ClVi85, D1-06}:

\begin{prop}\label{prop Lagrangian minimizer}
Let $x,\,y\in M$, $t>0$ and $C\in\R$ such that $h^t(x,y)<t\,C$.
Then any Lagrangian minimizer $\gamma$ for $h^t(x,y)$ is Lipschitz
continuous and satisfies $\|\dot\gamma\|_\infty\leqslant \kappa$,
where $\kappa$ is a constant only depending on
$C,\,\alpha_*,\,\beta_*$.
\end{prop}

\indent We recall some properties of $h^t$, see for instance
\cite{D1-06}.

\begin{prop}\label{prop h^t}
Let $t>0$. Then $h^t$ is locally Lipschitz continuous in $M\times
M$. More precisely, for every $r>0$ there exists
$K=K(r,\alpha_*,\beta_*)$ such that the map
\[
(x,y,t)\mapsto h^t(x,y)\quad\hbox{is $K$--Lipschitz continuous in
$C_r$},
\]
where $C_r:=\{(x,y,t)\in M\times M\times
(0,+\infty)\,:\,|x-y|<r\,t\,\}$.
\end{prop}

\medskip
\indent We remark that, for every $a\geqslant c$, the following
holds:
\begin{equation}\label{eq L}
L(x,q)\geqslant \max_{p\in Z_a(x)}\,\langle p,\,q\rangle
-H(x,p)\,\geqslant \sigma_a(x,q)-a\qquad\hbox{for every $(x,q)\in
M\times\R^N$,}
\end{equation}
yielding in particular $h^t(y,x)+a\,t\geqslant S_a(y,x)$ for every
$x,\,y \in M$. The next result can be proved by making use of
suitable reparametrization techniques, see \cite{DavSic05,
FaSic03}.

\begin{lemma}\label{lemma 4}
Let $a\geqslant c$. Then
\[
S_a(y,x)=\inf_{t>0} \Big(h^t(y,x)+at\Big)\qquad\hbox{for every
$x,y\in\ M$,}
\]
and the infimum is always reached when $a>c$.
\end{lemma}
\medskip

For every $t>0$, we define a function on $M$ as follows:
\begin{equation}\label{def S(t)}
\big(\S(t)u\big)(x)=\inf\left\{u\big(\gamma(0)\big)+\int_{-t}^0
L(\gamma,\dot\gamma)\,\dd s\ :\ \gamma\in W^{1,1}([-t,0];M),\
\gamma(0)=x \right\}
\end{equation}
where $u: M \to\R\cup\{+\infty\}$ is an initial datum satisfying
\begin{equation}\label{admissible u}
u(\cdot)\geqslant a|\cdot|+b\qquad\hbox{on $ M $}
\end{equation}
for some $a,\,b\in\R$. Any function of this kind will be called
{\em admissible initial datum} in the sequel.

The following properties hold:

\begin{prop}\label{prop S(t)}\
  \begin{itemize}
    \item[{\em (i)}]  For every admissible initial datum $u$,
    the  map $(t,x)\mapsto \big(\S(t)u\big)(x)$ is
    finite valued and locally Lipschitz in $(0,+\infty)\times M.$\smallskip
    \item[{\em (ii)}] $\big(\S(t)\big)_{t>0}$ is a semigroup, i.e. for every
    admissible initial datum $u$
    \[
    \S(t)\big(\S(s)u\big)=\S(t+s)u\qquad\hbox{for every $t,\,s>0$.\smallskip}
    \]
    \item[{\em (iii)}] $\S(t)$ is monotone and commutes with
    constants, i.e. for every admissible initial data $u,\, v$ and
    any $a\in\R$ we have
    \[
    u\leqslant v \quad \Longrightarrow \quad\S(t)u\leqslant
    \S(t)v\qquad\hbox{and}\qquad \S(t)(u+a)=\S(t)u+a.
    \]
    In particular, $\S(t)$ is weakly contracting, i.e.
    \[
    \|\S(t)u-\S(t)v\|_\infty\leqslant \|u-v\|_\infty.\smallskip
    \]
    \item[{\em (iv)}] If $u\in \D{Lip}(M)$, then the map $(t,x)\mapsto
\big(\S(t)u\big)(x)$ is
    Lipschitz continuous in $[0,+\infty)\times M$ and
    \[
    \lim_{t\to 0^+} \|\S(t)u-u\|_\infty=0.
    \]
  \end{itemize}
\end{prop}

The semigroup $\big(\S(t)\big)_{t>0}$ is called Lax--Oleinik
semigroup and \eqref{def S(t)} is termed Lax--Oleinik formula. The
relation with Hamilton--Jacobi equations is clarified by the next
results.

\begin{teorema}
Let $H$ be a convex Hamiltonian and let $L$ be its Fenchel
transform. Then, for every $u_0\in \D{Lip}(M)$, the Cauchy Problem
\begin{eqnarray}\label{Cauchy problem}
\begin{cases}
\partial_t u+H(x,Du)=0& \hbox{in $(0,+\infty)\times M$}\\
u(0,x)=u_0(x) &\hbox{on $M$}
\end{cases}
\end{eqnarray}
admits a unique viscosity solution $u(t,x)$ in
$\D{Lip}([0,+\infty)\times M)$. Moreover,
\[
u(t,x)=\big(\S(t)u_0\big)(x)\qquad\hbox{for every $(t,x)\in
(0,+\infty)\times M$}.
\]
\end{teorema}

With regard to the stationary equation \eqref{eq hja}, the
following characterization holds:

\begin{prop}\label{prop critical sol} Let $u$ be a continuous function on $M$. The following facts hold:
\begin{itemize}
    \item[{\em (i)}] $u$ is a subsolution of \eqref{eq hja} if and only
if $t\mapsto \S(t)u+a\,t$ is non decreasing;\smallskip
    \item[{\em (ii)}] $u$ is a solution of \eqref{eq hja} if and only
if $u\equiv S(t)u+a\,t$ for every $t>0$.
\end{itemize}
\end{prop}

\begin{dimo}
{\em (i)} If $u$ is a subsolution of \eqref{eq hja}, then for
every $x,\,y\in M$
\[
u(x)\leqslant u(y)+S_a(y,x)\leqslant u(y)+
h^t(y,x)+a\,t\quad\hbox{for every $t>0$,}
\]
hence
\[
u\leqslant \inf_{y\in M}\big( u(y)+ h^t(y,\cdot)+a\,t
\big)=\S(t)u+a\,t.
\]
This readily implies, by monotonicity of the semi--group,
\[
\S(h)u+a\,h\leqslant \S(t+h)u+a\,(t+h)\quad\hbox{for every $h>0$,}
\]
i.e. $t\mapsto\S(t)u+a\,t$ is non decreasing.

Conversely, if $t\mapsto\S(t)u+a\,t$ is non decreasing, then for
every fixed $x,\,y\in M$ we have
\[
u(x)\leqslant u(y)+h^t(y,x)+a\,t\qquad\hbox{for every $t>0$.}
\]
By taking the infimum in $t$ of the right--hand side term, we
obtain $u(x)-u(y)\leqslant S_a(y,x)$ for every $x,\,y\in M$ by
Lemma \ref{lemma 4}, i.e. $u$ is a subsolution of \eqref{eq
hja}.\smallskip

Assertion {\em (ii)} easily follows by noticing that $u$ is a
solution of \eqref{eq hja} if and only if $u(x)-at$ is a solution
of \eqref{Cauchy problem} with $u_0:=u$.
\end{dimo}

\medskip
We conclude this section by proving a result that we will need later in the paper.


\begin{lemma}\label{lemma Maxime}
Let $H$ be a strictly convex Hamiltonian and $u$ an admissible datum. Then, for each $x\in M$, the function $t \mapsto S(t)u(x) $ is locally semi-concave on $(0,+\infty)$. Moreover, the modulus of semi-concavity is locally uniform in $x$.    
\end{lemma}
\begin{proof}
For simplicity, we prove the assertion for $M=\T^N$. The proof in the general case goes along the same lines, but one has to localize the arguments. 

Let $I$ be an open interval compactly contained in $(0,+\infty)$. By compactness of $\T^N$ it is not hard to see that there exists a constant $C$ such that
\[
h^t(x,y)\leqslant C\,t\qquad\hbox{for every $x,y\in\T^N$ and $t\in I$}.
\]
Let $\kappa$ be the constant chosen according to Proposition \ref{prop Lagrangian minimizer}. Fix $x\in\T^N$, $t\in I$ and let 
$\gamma$ be a Lipschitz curve (of Lipschitz constant $\kappa$) verifying $\gamma(t)=x$ and such that
\[
\big(S(t)u\big)(x)=u\big(\gamma(0)\big)+\int_0^t L\big(\gamma(s),\dot\gamma(s)\big)\, \dd s.
\]
For $h$ such that $|h|<t/2$ we set $\gamma_h : [0,t+h] \to \T^N$ by
$$\gamma_h(s) = \gamma\Big( \frac{ts}{t+h}\Big),\qquad\hbox{$s\in [0,t+h]$}.$$
By definition of the Lax-Oleinik semigroup, we have the following obvious inequality:
$$\big(S(t+h)u\big)(x)\leqslant u\big(\gamma(0)\big)+\int_0^{t+h}L\big(\gamma_h(s),\dot\gamma_h(s)\big)\,\dd s.$$
Therefore, the following holds:
\begin{align*}
\big(S(t+h)u\big)(x)-\big(S(t)u\big)(x)&\leqslant\int_0^{t+h}L(\gamma_h,\dot\gamma_h)\,\dd s-\int_0^t L(\gamma,\dot\gamma)\,\dd s \nonumber\\
&=\int_0^t \Big(L\big(\gamma(s),\dot\gamma(s)\cdot\frac{t}{t+h}\big)\cdot \frac{t+h}{t}-L\big(\gamma(s),\dot\gamma(s)\big)\Big)\,\dd s.
\end{align*}
We make a Taylor expansion to obtain that
\begin{align*}
\big|L \big(\gamma(s),\dot\gamma(s)\cdot\frac{t}{t+h}\big)-L\big(\gamma(s),\dot\gamma(s)\big) & -\big \langle \frac{\partial L}{\partial q}\big(\gamma(s),\dot\gamma(s)\big),\, \dot\gamma(s)\big\rangle\big|  \nonumber \\
& \leqslant\frac{\kappa|h|}{t+h}\,\omega\Big(\frac{\kappa |h|}{t+h}\Big)\leqslant \frac{2 \kappa |h|}{t}\,\omega\Big(\frac{2 \kappa |h|}{t}\Big),
\end{align*}
where $\omega$ is a continuity modulus for $\partial L/\partial q$ in $\T^N\times B_{2\kappa}$. We deduce  that
\begin{align*}
\nonumber \big(S(t&+h)u\big)(x) -\big(S(t)u\big)(x)\\
\nonumber & \leqslant   \int_0^t \Big(\frac{h}{t} L\big(\gamma(s),\dot\gamma(s)\big) -\frac{h}{t} \big\langle\frac{\partial L}{\partial q}\big(\gamma(s),\,\gamma(s)\big),\, \dot\gamma(s)\big\rangle + 2\frac{\kappa|h|}{t}\omega\big(\frac{2\kappa|h|}{t}\big) \Big)\dd s \\
&=h\,p_t+2{\kappa|h|}\,\omega\Big(\frac{2\kappa |h|}{t}\Big),
\end{align*}
where 
\[
p_t=\displaystyle{\int\tagli}_0^{\,t} \Big(L\big(\gamma(s),\dot\gamma(s)\big)-\big \langle\frac{\partial L}{\partial q}\big(\gamma(s),\dot\gamma(s)\big),\,  \dot\gamma(s)\big\rangle\Big)\, \dd s.
\]
\end{proof}
\end{subsection}
\end{section}

\begin{section}{Nonregular weak KAM theory}\label{sez weak
KAM}

The purpose of this Section is to present the main results of weak
KAM Theory we are going to use in the sequel. This material is not
new. It is well known for Tonelli Hamiltonians, see \cite{Fathi},
while the extension to the non regular setting is either contained
in other papers or can be easily recovered from the results proved
therein. Nevertheless, it is less standard and it is not always
possible to give precise references. For the reader's convenience,
we provide here a brief presentation. Some proofs are postponed to
 Appendix \ref{appendix b}.\smallskip

\indent Throughout this Section, $M$ stands either for $\R^N$ or
for $\T^N$ and conditions (H1), (H2) and (H3) are
assumed.\smallskip

We focus our attention on the critical equation
\begin{equation}\label{eq hjc}
H(x,Du)=c\qquad\hbox{in $M$.}
\end{equation}

\noindent A subsolution, supersolution or solution of \eqref{eq
hjc} will be termed critical in the sequel. To ease notations, we
will moreover write $S$ and $\sigma$ in place of $S_c$ and
$\sigma_c$, respectively. Finally, by possibly considering $H-c$
instead of $H$, we will assume $c=0$.

We define the {\em Aubry set} $\A$ as
\[
\A:=\{y\in M\,:\,S(y,\cdot)\ \hbox{is a critical solution}\,\}.
\]
In the sequel, we will sometimes write $S_y$ to denote the
function $S(y,\cdot)$.

We will assume that the following holds
\begin{itemize}
  \item[($\A$)]\qquad $\A$ is nonempty.
\end{itemize}
This condition is always fulfilled when $M$ is compact, but it may
be false in the non compact case.\smallskip

We define the set $\E$ of equilibrium points as
\[
\E:=\{y\in M\,:\,\min_p H(y,p)=0\,\}.
\]
This set may be empty, but if not it is a closed subset of the
Aubry set $\A$.

%

Next, we define a family of curves, called {\em static}. In the
next Section we will investigate the behavior of the critical
subsolutions on such curves.
\begin{definition}\label{critical}
A curve $\gamma$ defined on an interval $J$ is called  {\em
static} if
\[
S\big(\gamma(t_1),\gamma(t_2)\big) =\int_{t_1}^{t_2}
L(\gamma,\dot\gamma ) \,\dd s =-
S\big(\gamma(t_2),\gamma(t_1)\big)
\]
for every $t_1$, $t_2$ in $J$ with $t_2 > t_1$.
\end{definition}

We first show that static curves are always contained in the Aubry
set.

\begin{lemma}\label{DL1}
Let $\gamma$ be a static curve defined on some interval $J$. Then
$\gamma$ is contained in the Aubry set and satisfies
\begin{equation}\label{eq lagrangian parametrization}
L\big(\gamma(s),\dot\gamma(s)\big)=\sigma\big(\gamma(s),\dot\gamma(s)\big)\qquad\hbox{for
a.e. $s\in J$}.
\end{equation}
\end{lemma}
\begin{dimo}  The definition of the semidistance $S$ and
inequality \eqref{eq L} with $a=c$ readily implies that $\gamma$
enjoys \eqref{eq lagrangian parametrization}.

Let us prove that $\gamma$ is contained in the Aubry set.  If
$\gamma$ is a steady curve, i.e. $\gamma(t)=y$ for every $t\in J$,
then for $(a,b)\subset J$ we get
\[
    (b-a)\,L(y,0)    =
    \int_a^b L(\gamma,\dot\gamma) \,\dd s
    =
    S(y,y)
    =0,
\]
yielding that $y\in\E\subseteq \A$ for $L(y,0)=-\min_{\R^N}
H(y,\cdot)$.

Let us then assume that $\gamma$ is nonsteady. We want to prove
that, for every fixed $t\in J$, the point $y:=\gamma(t)$ belongs
to $\A$, i.e. that $S(y,\cdot)$ is a critical solution on $ M$. Of
course, we just need to check that $S(y,\cdot)$ is a supersolution
of \eqref{eq hjc} at $y$, by Proposition \ref{prop S}. To this
purpose, choose a point $z\in\gamma(J)$ with $z\not=y$. Since
$\gamma$ is static, we have
\[
S(y,z)+S(z,y)=0.
\]
This and the triangular inequality imply that the function
$w(\cdot)=S(y,z)+S(z,\cdot)$ touches $S(y,\cdot)$ from above at
$y$, hence $D^- S_y(y)\subseteq D^- w(y)$. Since $w$ is a
viscosity solution in $ M\setminus\{z\}$ we derive
\[
H(y,p)\geqslant  0  \qquad\hbox{for every $p\in D^-{S_y}(y)$,}
\]
that is, $S_y$ is a supersolution of \eqref{eq hjc} at $y$ and so
a critical solution on $ M$.
\end{dimo}

\bigskip
The next result states that static curves fully cover the Aubry
set.

\begin{teorema}\label{DT1} Let $y\in\A$, then there exists a static curve $\eta$ defined on
$\R$ with  $\eta(0)=y$.
\end{teorema}

This result is proved in \cite{DavSic05} by exploiting some ideas
contained in \cite{FaSic03}. A more concise and self--contained
proof of this fact is proposed in the Appendix \ref{appendix b}.

\medskip
We denote by  $\K$ the family of all static curves defined on
$\R$, and by $\K(y)$ the subset of $\K$ made up by those equaling
$y$ at $t=0$.\\

The {\em Peierls barrier} is the function $h:M\times M\to\R$
defined by
\begin{equation}\label{def h}
h(x,y)=\liminf_{t\to +\infty} h^t(x,y).
\end{equation}

The following holds:

\begin{teorema}\label{teo A}
$\quad \A=\{y\in M\,:\,h(y,y)=0\,\}.$
\end{teorema}

\begin{dimo}
Take $y\in M$ such that $h(y,y)=0$ and set $u(\cdot):=S(y,\cdot)$.
We want to prove that $u$ is a critical solution in $M$;
equivalently, by Proposition \ref{prop S}, that $u$ is a critical
supersolution at $y$. To this purpose, we first note that, since
$u$ is a critical subsolution on $M$, the functions $\S(t)u$ are
increasing in $t$, see  Proposition \ref{prop critical sol}, and
equi--Lipschitz in $x$ for $u$ is Lipschitz continuous, see
Proposition \ref{prop S(t)}. Let us set
\[
v(x)=\sup_{t>0} \big(\S(t)u\big)(x)=\lim_{t\to
+\infty}\big(\S(t)u\big)(x)\qquad\hbox{for every $x\in M$}.
\]
According to what was remarked above, $v\geqslant u$. Furthermore,
$v$ is Lipschitz continuous provided it is finite everywhere, or,
equivalently, at some point. We claim that $v(y)=u(y)$.

\indent Indeed, let $(t_n)_{n\in \mathbb{N}}$ be a diverging
sequence such that $\lim_{n\in \N} h^{t_n}(y,y)=h(y,y)=0$. By definition
of $\S(t)$ we have
\[
\big(\S(t_n)u\big)(y)\leqslant u(y)+h^{t_n}(y,y)\qquad\hbox{for
each $n\in\N$,}
\]
hence
\[
v(y)
    =
\lim_{n\to +\infty}\big(\S(t_n)u\big)(y)
    \leqslant
\lim_{n\to +\infty} \big( u(y)+h^{t_n}(y,y)\big) =u(y),
\]
as it was claimed. This also implies that $v$ touches $u$ from
above at $y$, yielding $D^- u(y)\subseteq D^- v(y)$. Furthermore,
$v$ is a critical solution since it is a fixed point of the (continuous)
semigroup $\S(t)$, see  Proposition \ref{prop critical sol}, in
particular it is a critical supersolution at $y$. Collecting the
information, we conclude that
\[
H(y,p)\geqslant 0  \qquad\hbox{for every $p\in D^-u(y)$,}
\]
finally showing that $u$ is a critical supersolution at
$y$.\smallskip

Let us prove the opposite inclusion. Take $y\in\A$. To prove that
$h(y,y)=0$, it will be enough, in view of Lemma \ref{lemma 4}, to
find a diverging sequence $(t_n)_{n\in \mathbb{N}}$ such that
$\liminf_{n\in \N} h^{t_n}(y,y)= 0$.

To this purpose, let $\eta\in\K(y)$. Then
\[
-S\big(\eta(n),y\big)=\int_0^n L(\eta,\dot\eta)\,\dd s =S\big(y,\eta(n)\big)
\]
for each $n\in\N$. By Lemma \ref{lemma 4} there exist $s_n>0$ such
that
\[
S\big(\eta(n),y\big)\leqslant h^{s_n}\big(\eta(n),y\big)<
S\big(\eta(n),y\big)+\frac{1}{n}.
\]
By definition of $h^t$ we get
\[
    h^{n+s_n}(y,y)
    \leqslant
    h^{n}\big(y,\eta(n)\big)+h^{s_n}\big(\eta(n),y\big)
    <
    S\big(y,\eta(n)\big)+S\big(\eta(n),y\big)+\frac{1}{n}
    =\frac{1}{n},
\]
and the assertion is proved by taking $t_n:=n+s_n$.\bigskip
\end{dimo}

Here and in the remainder of the paper, by $\check H$ we will
denote the Hamiltonian defined as
\[
\check H(x,p):=H(x,-p)\qquad\hbox{for every $(x,p)\in
M\times\R^N$.}
\]
The following holds:
\begin{prop}\label{prop H check}
The Hamiltonians $H$ and $\check H$ have the same critical value
and the same Aubry set.
\end{prop}

\begin{dimo}
The fact that $H$ and $\check H$ have the same critical value
immediately follows from the definition in view of Remark \ref{oss
subsol}. Furthermore, the Peierls barrier $\check h$ associated
with $\check H$ enjoys $\check h(x,y)=h(y,x)$ for every $x,\,y\in
M$. Hence $H$ and $\check H$ have the same Aubry set in view of
Theorem \ref{teo A}.
\end{dimo}

\medskip
We end this section by proving some important properties of the
Peierls barrier.

\begin{prop}\label{prop h}
Under assumption ($\A$) the following properties hold:
\begin{itemize}
    \item[{\em (i)}]\quad $h$ is finite valued and Lipschitz
continuous.\smallskip
    \item[{\em (ii)}]\quad  If $v$ is a critical subsolution, then
    $h(y,x)\geqslant v(x)-v(y)$ for every $x,y\in M$.\smallskip
    \item[{\em (iii)}]\quad  For every $x,y,z\in M$ and $t>0$
    \[
    h(y,x)\leqslant h(y,z)+h^t(z,x)\quad\hbox{and} \quad h(y,x)\leqslant h^t(y,z)+h(z,x).
    \]
    \quad In particular, $h(y,x)\leqslant h(y,z)+h(z,x)$.\smallskip
    \item[{\em (iv)}]\quad  $h(x,y)=S(x,y)$ \  if either $x$ or $y$
    belong to $\A$.\smallskip
    \item[{\em (v)}]\quad $h(y,\cdot)$ is a critical solution  for every fixed $y\in
    M$.\smallskip
%
\end{itemize}
Furthermore, when $M$ is compact and condition (H2)$'$ is assumed,
we have
\[
h^t \ \underset{t\to +\infty}\ucv \  h\qquad\hbox{in $M\times M$.}
\]
\end{prop}

\begin{dimo}
{\em (i)} Let $K_1$ be the constant given by Proposition \ref{prop
h^t} with $r=1$. It is easily seen that for every bounded open set
$V\subset M\times M$ there exists $t_V$ such that the functions
$\{h^t\,:\,t\geqslant t_V\,\}$ are $K_1$--Lipschitz continuous in
$V$. Moreover we already know, by Theorem \ref{teo A}, that
$h(y,y)=0$ for every $y\in\A$. This implies that $h$ is finite
valued and Lipschitz--continuous on the whole $M\times M$.

Items {\em (ii)} and {\em (iii)} follow directly from the
definition of $h$ and from assertion {\em (ii)} in Proposition
\ref{prop h^t}.

\indent{\em (iv)} Let us assume, for definiteness, that $y\in\A$.
Let $(t_n)_{n\in \N}$ be a diverging sequence such that $0\leqslant
h^{t_n}(y,y)<1/n$. Then for every $t>0$ and $n\in\N$
\[
S(x,y)\leqslant h^{t+t_n}(x,y) \leqslant h^{t}(x,y)+h^{t_n}(y,y)
\leqslant h^t(x,y)+\frac{1}{n},
\]
yielding
\[
S(x,y)\leqslant\liminf_{t\to +\infty}\, h^t(x,y) \leqslant
\inf_{t>0}\, h^t(x,y)=S(x,y)
\]
in view of Lemma \ref{lemma 4}.

\indent{\em (v)} By Proposition \ref{prop critical sol}--{\em
(ii)}, it suffices to prove that $\S(t)h_y =h_y$ for every fixed
$t>0$ and $y\in M$, where $h_y$ denotes the function $h(y,\cdot)$.
First notice that, by {\em (iii)} and Lemma \ref{lemma 4},
\[
h_y(x)-h_y(z)\leqslant \inf_{t>0} h^t(z,x)=S(z,x),
\]
that is, $h_y$ is a critical subsolution. By Proposition \ref{prop
critical sol}--{\em (i)}, that implies $\S(t)h_y \geqslant h_y$.

Let us prove the reverse inequality.  For any fixed $x\in M$, pick
a diverging sequence $(t_n)_{n\in \N}$ with $t_n>t$ for every
$n\in\N$ and a family of curves $\gamma_n:[-t_n,0]\to M$
connecting $y$ to $x$ such that $h^{t_n}(y,x)=\int_{-t_n}^0
L(\gamma_n,\dot\gamma_n)\,\dd s$ and
\begin{equation}\label{eq hh2}
\lim_{n\to +\infty}\int_{-t_n}^0 L(\gamma_n,\dot\gamma_n)\,\dd
s=h(y,x).
\end{equation}
The functions $h^{t_n}$ are equi--Lipschitz, see Proposition
\ref{prop h^t}. This yields, by Proposition \ref{prop Lagrangian
minimizer}, that the curves $\gamma_n$ are equi--Lipschitz. Up to
extraction of a subsequence, we can then assume that there is a
curve $\gamma:[-t,0]\to M$ such that
\[
\gamma_n\ucv\gamma\quad\hbox{in $[-t,0]$}\qquad\hbox{and}\qquad
\dot\gamma_n\weakcv\dot\gamma\quad\hbox{in}\quad
L^1\big([-t,0];\R^N\big).
\]
Set  $z=\gamma(-t)$. By a classical semi--continuity result of the
Calculus of Variations \cite{BuGiHi98}, we have
\begin{multline*}
h_y(x)= \liminf_{n\to +\infty}\int_{-t_n}^0
L(\gamma_n,\dot\gamma_n)\,\dd s
 \\ \geqslant
 \liminf_{n\to +\infty}\int_{-t_n}^{-t} L(\gamma_n,\dot\gamma_n)\,\dd
 s\,+\,
 \liminf_{n\to +\infty}\int_{-t}^0 L(\gamma_n,\dot\gamma_n)\,\dd s
\\ \geqslant h(y,z)+  \int_{-t}^0 L(\gamma,\dot\gamma)\,\dd s \geqslant \big(\S(t)h_y\big)(x).
\end{multline*}

Last, let us show that $h^t$ uniformly converges to $h$ for $t\to
+\infty$ when $M$ is compact and condition (H2)$'$ is assumed. Let
$y\in M$ be fixed. Then the convergence of $h^t(y,\cdot)$ to
$h(y,\cdot)$ is actually uniform, in view of the asymptotic convergence results proved in \cite{BaSo00, DavSic05}
and of the equality $h^t(y,\cdot)=\S(t-1)u$ with $u=h^1(y,\cdot)$. The assertion follows from the fact that  
$y$ was arbitrarily chosen in $M$ and 
the functions $\{h^t\,:\,t\geqslant 1\}$ are equi--Lipschitz  in
$M\times M$  in view of Proposition \ref{prop h^t}.
\end{dimo}
\  \\
\end{section}

\begin{section}{Differentiability properties of critical
subsolutions}\label{sez differentiability}

The purpose of this Section is to prove some differentiability
properties of critical subsolutions on the Aubry set. These
results will be exploited in the subsequent section to obtain some
information for commuting Hamiltonians.\smallskip

Let us consider, for any fixed $t>0$, the locally Lipschitz
function defined on $M$ as
\[
\big(\S(t)u\big)(\cdot)=\inf_{z\in M}\left( u(z)+h^t(z,\cdot)
\right),
\]
where $u$ is an admissible initial datum. If the latter is
additionally assumed continuous, then the  infimum is actually a
minimum, and, as previously noticed, for every fixed $y\in M$
there exists a Lipschitz curve $\gamma:[-t,0]\to M$ with
$\gamma(0)=y$ such that
\[
\big(\S(t)u\big)(y)=u\big(\gamma(-t)\big)+\int_{-t}^0
L(\gamma,\dot\gamma)\,\dd s.
\]
As first step in our analysis, we prove some differentiability
properties of $\S(t)u$ at $y$ and of $u$ at $\gamma(-t)$ in terms
of $\gamma$, thus generalizing to this setting some known results
in the regular case, see \cite{Fathi}.

We start by dealing with the case when the Hamiltonian is
independent of $x$. We need a lemma first.\par

\begin{lemma}\label{lemma segment}
Let $H$ be a strictly convex Hamiltonian that does not depend on
$x$. Then any (Lipschitz) Lagrangian minimizer $\gamma:[-t,0]\to
M$ with $t>0$ satisfies
\begin{equation}\label{eq segment}
D_q L\big(\dot\gamma(s)\big)=D_q
L\Big(\frac{\gamma(0)-\gamma(-t)}{t}\Big)\qquad\hbox{for a.e.
$s\in [-t,0]$,}
\end{equation}
with equality holding for every $s$ if $\gamma$ is of class $C^1$.
\end{lemma}

\begin{oss}\label{oss segment}
We remark for later use that, since equality \eqref{eq segment}
holds for almost every $s\in [-t,0]$, then it holds in particular
for every $s$ that is both a differentiability point of $\gamma$
and a Lebesgue point of $D_q L\big(\dot\gamma(\cdot)\big)$ in $[-t,0]$.
\end{oss}

\begin{dimo}
Let us set
\[
\displaystyle{v:=\frac{\gamma(0)-\gamma(-t)}{t}}\qquad\hbox{and}
\qquad\eta(s)=\gamma(0)+sv.
\]
It is easy to see, by the convexity of $L$ and Jensen's
inequality, that
\[
\int_{-t}^0 L(\dot\gamma)\,\dd s\geqslant t\, L(v)=\int_{-t}^0
L(\dot\eta)\,\dd s,
\]
while the converse inequality is true since $\gamma$ is a
Lagrangian minimizer. By exploiting the convexity of $L$ again, we
get
\begin{equation}\label{eq Jensen}
L(q)\geqslant L(v)+\langle D_q L(v),q-v \rangle\qquad\hbox{for
every $q\in\R^N$.}
\end{equation}
On the other hand,
\[
\int_{-t}^0 L\big(\dot\gamma(s)\big)\,\dd s= t\,L(v)=\int_{-t}^0 \Big(
L(v)+\langle D_q L(v),\dot\gamma(s)-v \rangle\Big)\,\dd s,
\]
meaning that we have an equality in \eqref{eq Jensen} at
$\dot\gamma(s)$ for a.e. $s\in [-t,0]$. Equality \eqref{eq
segment} follows by differentiability of $L$.
\end{dimo}

\medskip

\begin{prop}\label{prop Fathi}
Let $H$ be a strictly convex Hamiltonian that does not depend on
$x$. Let $u$ be an admissible initial datum and $\gamma:[-t,0]\to
M$ a Lipschitz continuous curve such that $\gamma(0)=y$ and
\[
\big(\S(t)u\big)(y)=u\big(\gamma(-t)\big)+\int_{-t}^0
L\big(\dot\gamma(s)\big)\,\dd s
\]
for some $t>0$ and $y\in M$.  Then
\[
D_q L\Big(\frac{\gamma(0)-\gamma(-t)}{t}\Big)\in
{D}^{+}\big(\S(t)u\big)(y)\quad \hbox{and}\quad D_q
L\Big(\frac{\gamma(0)-\gamma(-t)}{t}\Big)\in
{D}^{-}u\big(\gamma(-t)\big).
\]
\end{prop}

\begin{dimo}
To ease notations, we set
\[
v:=\frac{\gamma(0)-\gamma(-t)}{t}
\]
and denote by $z$ the point $\gamma(-t)$. Let us first prove that
$D_q L\left({v}\right)\in {D}^{+}\big(\S(t)u\big)(y)$. According
to the proof of Lemma \ref{lemma segment}, it is enough to prove
the assertion when $\gamma$ is the segment joining $z$ to $y$. For
every $x\in M$, we define a curve $\gamma_x:[-t,0]\to M$ joining
$z$ to $x$ by setting $\gamma_x(s)=\gamma(s)+(s+t)(x-y)/t$. Let
\[
\varphi(x):=u(z)+\int_{-t}^0 L(\dot\gamma_x)\,\dd
s,\qquad\hbox{$x\in M$.}
\]
Then $\big(\S(t)u\big)(\cdot)\leqslant\varphi(\cdot)$ with
equality holding at $y$. It is easy to see, using the local
Lipschitz character of $L$, that $\varphi$ is locally Lipschitz
continuous. We want to show that $D_q L(v)\in {D}^{+}\varphi(y)$,
which clearly implies the assertion as ${D}^{+}\varphi(y)\subseteq
{D}^{+}\big(\S(t)u\big)(y)$.

By the standard result of differentiation under the integral sign, the function
$\varphi$ is in fact $C^1$ and we may compute its differential at $y$ by the following formula:
\[{D}\varphi (y)
= \Big(\int_{-t}^0 \frac{\partial}{\partial x} L(\dot\gamma_x)\,
\dd s\Big)_{{\mbox{\Large $|$}}_{x=y}} =D_q L\left(v\right).
\]

\smallskip

Let us now prove that $D_q L(v)\in
{D}^{-}u(z)$.\smallskip\\
For every $x\in M$, we define a curve $\eta_x:[-t,0]\to M$ joining
$x$ to $y$ by setting $\eta_x(s):=\gamma(s)+s\,(z-x)/t$. Let
\[
\psi(x):=-\int_{-t}^0 L(\dot\eta_x)\,\dd
s+\big(\S(t)u\big)(y),\qquad\hbox{$x\in M$.}
\]
Then $\psi(\cdot)\leqslant u(\cdot)$ with equality holding at $z$.
We want to show that $D_q L(v)\in {D}^{-}\psi(z)$, which is enough
to conclude as ${D}^{-}\psi(z)\subseteq {D}^{-} u(z)$. Arguing as
above, we actually see that $\psi$ is in fact $C^1$ and
\[{D}\psi (z)=D_q L\left(v\right).
\]
This concludes the proof.
\end{dimo}
\ \\

We proceed to show a more general version of the previous result.

\begin{prop}\label{prop Fathi+}
Let $H$ be a strictly convex Hamiltonian and $u$ an admissible
initial datum. Let $\gamma:[-t,0]\to M$ be a Lipschitz continuous
curve with $\gamma(0)=y$ such that
\[
\big(\S(t)u\big)(y)=u\big(\gamma(-t)\big)+\int_{-t}^0
L\big(\gamma(s),\dot\gamma(s)\big)\,\dd s
\]
for some $t>0$ and $y\in M$. The following holds:
\begin{itemize}
  \item[\em (i)] if $\ 0$ is a differentiability point for $\gamma$
  and a Lebesgue point for $D_q
  L\big(\gamma(\cdot),\dot\gamma(\cdot)\big)$, then $D_q L\big(\gamma(0),\dot\gamma(0)\big)\in
{D}^{+}\big(\S(t)u\big)(y)$.\smallskip
\item[\em (ii)] Assume $u\in\D{Lip}(M)$. If $\ -t$ is a
differentiability point for $\gamma$
  and a Lebesgue point for $D_q
  L\big(\gamma(\cdot),\dot\gamma(\cdot)\big)$, then $D_q L\big(\gamma(-t),\dot\gamma(-t)\big)\in
{D}^{-}u\big(\gamma(-t)\big)$.
\end{itemize}
\end{prop}

\begin{dimo}
Let us choose an $R>1$ sufficiently large in such a way that
$\|\dot\gamma\|_{\infty}\leqslant R$ and
$\gamma\big([-t,0]\big)\subseteq B_R$. To ease notations, in the
sequel we will call $z$ the point $\gamma(-t)$.\smallskip

Let $\omega:\R_+\to\R_+$ be a  modulus such that
\[
|L(x,q)-L(y,q)|\leqslant\omega(|x-y|)^2\qquad\hbox{for every
$x,y\in B_{2R}$ and $q\in B_{2R}$.}
\]
If $\omega(h)=O(h)$ then $L(x,q)=L(q)$ on $B_{2R}\times B_{2R}$,
and the assertion follows from  Proposition \ref{prop Fathi} when
$\gamma$ is the segment joining $z$ to $y$, and from Remark
\ref{oss segment} when $\gamma$ is any Lipschitz continuous
minimizer.\smallskip

Let us then assume $\omega(h)/h$ is unbounded. Without loss of
generality, we may require $\omega$ to be concave, in particular
\[
\frac{\omega(h)}{h}\to+\infty\qquad\hbox{as $h\to 0^+$.}
\]
Let $\delta:(0,+\infty)\to(0,+\infty)$ be such that
\[
\delta(h)\,\omega(h)=h\qquad\hbox{for every $h>0$,}
\]
i.e.
\[
\delta(h):=\frac{h}{\omega(h)}\qquad\hbox{for every
$h>0$}.\smallskip
\]
\indent Since $\S(t)u$ is Lipschitz in $ M$, to prove assertion
{\em (i)}  it is in fact enough to show that the following
inequality holds for every $\xi\in\partial B_{R}$:
\begin{equation}\label{claim prima}
\big(\S(t)u\big)(y+h\xi)-\big(\S(t)u\big)(y)\leqslant
h\,\big\langle{D_q
L}\big(\gamma(0),\dot\gamma(0)\big),\xi\big\rangle+o(h)\qquad\hbox{for $h\to
0^+$.}
\end{equation}
To this purpose, for every $h\in [0,1]$ and for every $\xi\in
\partial B_1$ we define a Lipschitz curve $\gamma_{h\xi}:[-t,0]\to
M$ joining $z$ to $y+h\xi$ by setting
\begin{eqnarray*}
\gamma_{h\xi}(s):=
\begin{cases}
\gamma(s) & \hbox{if $s\in [-t,-\delta(h)]$}\\
\gamma(s)+\omega(h)\big(\delta(h)+s\big)\xi& \hbox{if
}s\in[-\delta(h),0].
\end{cases}
\end{eqnarray*}
By definition of $\big(\S(t)u\big)$, we get
\begin{eqnarray}\label{eq intermediate}
  && \big(\S(t)u\big)(y+h\xi)-\big(\S(t)u\big)(y)
    \leqslant
    \int_{-\delta(h)}^0 \big( L(\gamma_{h\xi},\dot\gamma_{h\xi})
    -  L(\gamma,\dot\gamma) \big)\,\dd t\\
    &&\qquad\quad=
    \underbrace{\int_{-\delta(h)}^0 \big( L(\gamma_{h\xi},\dot\gamma_{h\xi})
    -  L(\gamma,\dot\gamma_{h\xi}) \big)\,\dd t}_{A}
    +
    \underbrace{\int_{-\delta(h)}^0 \big( L(\gamma,\dot\gamma_{h\xi})
    -  L(\gamma,\dot\gamma) \big)\,\dd t}_{B}.\nonumber
\end{eqnarray}
For $h$ small enough we have
\begin{eqnarray*}
&&|\gamma_{h\xi}(t)-\gamma(t)|\leqslant h<R \qquad\qquad\hbox{for
every
$t\in[-\delta(h),0]$,}\\
&&|\dot\gamma_{h\xi}(t)|=|\dot\gamma(t)+\omega(h)\xi|<2R\,\quad\hbox{for
a.e. $t\in[-\delta(h),0]$,}
\end{eqnarray*}
hence
\[
|L(\gamma_{h\xi},\dot\gamma_{h\xi})-
L(\gamma,\dot\gamma_{h\xi})|\leqslant\omega(h)^2\quad\hbox{for
a.e. $t\in[-\delta(h),0]$}.
\]
This yields
\begin{equation}\label{eq A}
 A\leqslant\delta(h)\,\omega(h)^2=h\,\omega(h).
\end{equation}
To evaluate $B$, we use the Taylor expansion of
$L(\gamma,\dot\gamma_{h\xi})$ to get
\[
    L\big(\gamma,\dot\gamma+\omega(h)\xi\big)
    \leqslant
    L(\gamma,\dot\gamma)+\omega(h)\,\langle D_q L(\gamma,\dot\gamma),\xi\rangle
    +\omega(h)\,\Theta\big(\omega(h)\big)
\]
for a.e. $t\in[-\delta(h),0]$, where $\Theta$ is a continuity
modulus for $ D_q L$ on $B_{2R}\times B_{2R}$. From this we obtain
\begin{eqnarray*}
    B
    &\leqslant&
    \omega(h)\int_{-\delta(h)}^0\langle D_q L(\gamma,\dot\gamma),\xi\rangle\,\dd t
    +\delta(h)\,\omega(h)\Theta\big(\omega(h)\big)\\
    &\leqslant&
    h\,\big\langle D_q L\big(\gamma(0),\dot\gamma(0)\big),\xi\big\rangle+
    h\,\displaystyle{\int\tagli}_{-\delta(h)}^0
   \big | D_q L(\gamma,\dot\gamma)- D_q
    L\big(\gamma(0),\dot\gamma(0)\big)\big|\,\dd t
    +
    h\,\Theta\big(\omega(h)\big),
\end{eqnarray*}
i.e.
\begin{equation}\label{eq B}
B\leqslant h\,\big\langle D_q
L\big(\gamma(0),\dot\gamma(0)\big),\xi\big\rangle+o(h)\smallskip
\end{equation}
by recalling that $t=0$ is a Lebesgue point for $ D_q
L\big(\gamma(\cdot),\dot\gamma(\cdot)\big)$. Relations \abra{eq A} and
\abra{eq B} together with \abra{eq intermediate} finally give
\abra{claim prima}.\smallskip

To prove {\em (ii)}, it suffices to show, by the Lipschitz
character of $u$, that for every fixed $\xi\in\partial B_{1}$
\begin{equation}\label{claim seconda}
u(y+h\xi)-u(y)\geqslant h\,\big\langle{ D_q
L}\big(\gamma(0),\dot\gamma(0)\big),\xi\big\rangle+o(h)\qquad\hbox{for $h\to
0^+$.}
\end{equation}
To this purpose, for every $h\in [0,1]$ and for every $\xi\in
\partial B_1$ we define a Lipschitz curve $\eta_{h\xi}:[-t,0]\to
M$ joining $z+h\xi$ to $y$ by setting
\begin{eqnarray*}
\eta_{h\xi}(s):=
\begin{cases}
\gamma(s)+\omega(h)\big(\delta(h)-t-s\big)\xi& \hbox{if
}s\in[-t,\,-t+\delta(h)]\\
\gamma(s) & \hbox{if $s\in [-t+\delta(h),\,0]$}.
\end{cases}
\end{eqnarray*}
By definition of $\big(\S(t)u\big)(y)$, we get
\begin{eqnarray*}
  && u(z+h\xi)-u(z)
    \geqslant
    \int_{-t}^{-t+\delta(h)} \big( L(\gamma,\dot\gamma)
    -
    L(\eta_{h\xi},\dot\eta_{h\xi})\big)\,\dd t\\
    &&\qquad =
    \underbrace{\int_{-t}^{-t+\delta(h)} \big( L(\gamma,\dot\gamma)
    -  L(\gamma,\dot\eta_{h\xi}) \big)\,\dd t}_{A'}
    +
    \underbrace{\int_{-\delta(h)}^0 \big( L(\gamma,\dot\eta_{h\xi})
    -  L(\eta_{h\xi},\dot\eta_{h\xi}) \big)\,\dd t}_{B'}.\nonumber
\end{eqnarray*}
To evaluate $B'$, we argue as above to get $B'\geqslant
-h\omega(h)$. To evaluate $A'$, we use the Taylor expansion of
$L(\gamma,\dot\eta_{h\xi})$ to get
\[
    L\big(\gamma,\dot\gamma-\omega(h)\xi\big)
    \leqslant
    L(\gamma,\dot\gamma)-\omega(h)\,\langle D_q L(\gamma,\dot\gamma),\xi\rangle
    +\omega(h)\,\Theta\big(\omega(h)\big)
\]
for a.e. $t\in[-\delta(h),0]$. Arguing as above we finally get
\[
A'\geqslant h\,\big\langle D_q
L\big(\gamma(0),\dot\gamma(0)\big),\xi\big\rangle+o(h),\smallskip
\]
and \eqref{claim seconda} follows.
\end{dimo}

\ \smallskip\\
\indent We now exploit the information gathered to deduce some
differentiability properties of critical subsolutions. In what
follows, we stress the fact that we have assumed the critical
value $c$ to be equal to $0$, which is not restrictive up to the
addition of a constant to the Hamiltonian.

We start by recalling some results proved in previous works. We
underline that the compactness of $M$, which is assumed in these
papers, does not actually play any role for the results we are
about to state. The first one has been proved in \cite{FaSic03}.

\begin{prop}\label{prop FatSic}
Let $H$ be a convex Hamiltonian. For every $y\in M\setminus\A$ the
set $Z_0(y)$ has nonempty interior and
\[
D^-S_y(y)=Z_0(y).
\]
In particular, $S_y$ is not differentiable at $y$.
\end{prop}

Therefore, critical subsolutions are in general not differentiable
outside the Aubry set. The situation is quite different on it. A
fine result proved in \cite{FaSic03} shows that, when $H$ is
locally Lipschitz--continuous in $x$ and condition (H2)$'$ is
assumed, all critical subsolutions are (strictly) differentiable
at any point of the Aubry set, and have the same gradient. These
results are based upon some semiconcavity estimates which, in
turn, depend essentially on the Lipschitz character of the
Hamiltonian in $x$. Something analogous still survives in the case
of a purely continuous and convex Hamiltonian by looking at the
behavior of the critical subsolutions on static curves, see
\cite{DavSic05}.

\begin{teorema}\label{prop 1}
Let $H$ be a convex Hamiltonian and $\gamma\in\K$. Then there
exists a negligible set $\Sigma\subset\R$ such that, for any
critical subsolution $u$, the map $u\comp\gamma$ is differentiable
on $\R\setminus\Sigma$ and satisfies
\begin{equation}\label{claim derive}
\frac{\dd}{\dd t}\left(u\comp\,\gamma\right)(t_0)
=\sigma\big(\gamma(t_0),\dot\gamma(t_0)\big)\qquad\hbox{whenever
$t_0\in\R\setminus\Sigma$.}
\end{equation}
\end{teorema}

\vspace{2ex}

Here we want to strengthen Theorem \ref{prop 1} by proving that,
when condition (H2)$'$ is assumed, any critical subsolution is
actually differentiable at $\hh^1$--a.e. point of $\gamma(\R)$. We
give a definition first.

\begin{definition}
Let $\gamma$ be an absolutely continuous curve defined on $\R$. We
will denote by $\Sigma_\gamma$ the negligible subset of $\R$ such
that $\R\setminus \Sigma_\gamma$ is the following set:
\[
\big\{t\in \R\,:\,\hbox{$t$ is a differentiability point of $\gamma$
and a Lebesgue point of $D_q
L\big(\gamma(\cdot),\dot\gamma(\cdot)\big)$}\,\big\}.
\]
\end{definition}
\vspace{2ex}

\begin{teorema}\label{teo diff}
Let $H$ be a strictly convex Hamiltonian. Then, for any
$\gamma\in\K$, every critical subsolution $u$ is differentiable at
$\gamma(t_0)$ for any $t_0\in\R\setminus\Sigma_\gamma$, and we
have
\begin{equation}\label{claim derive2}
Du\big(\gamma(t_0)\big)={D_q
L}\big(\gamma(t_0),\dot\gamma(t_0)\big)\qquad\hbox{for every
$t_0\in\R\setminus\Sigma_\gamma$.}\medskip
\end{equation}
\end{teorema}
\begin{dimo}
%
Fix $t_0\in\R\setminus\Sigma$. As $u$ is a critical subsolution,
it is easily seen that
\[
\big(\S(t_0)u\big)(x)\geqslant u(x)\qquad\hbox{for every $x\in
M$},
\]
with equality holding at $\gamma(t_0)$ since
\[
\big(\S(t_0)u\big)\big(\gamma(t_0)\big)\leqslant u\big(\gamma(0)\big)+\int_0^{t_0}
L(\gamma,\dot\gamma)\,\dd s=u\big(\gamma(t_0)\big).
\]
By this and by Proposition \ref{prop Fathi+} we obtain
\[
D_q L\big(\gamma(t_0),\dot\gamma(t_0)\big)\in
{D}^{+}\big(\S(t_0)u\big)\big(\gamma(t_0)\big)\subseteq
{D}^{+}u\big(\gamma(t_0)\big)
\]
Analogously
\[
\big(\S(t_0+1)u\big)\big(\gamma(t_0+1)\big)=u\big(\gamma(t_0)\big)+\int_{t_0}^{t_0+1}
L(\gamma,\dot\gamma)\,\dd s,
\]
and by Proposition \ref{prop Fathi+} we have
\[
D_q L\big(\gamma(t_0),\dot\gamma(t_0)\big)\in {D}^{-}u\big(\gamma(t_0)\big).
\]
Then  $u$ is differentiable at $\gamma(t_0)$ and
$Du\big(\gamma(t_0)\big)=D_q L\big(\gamma(t_0),\dot\gamma(t_0)\big)$, as it was to
be shown.
\end{dimo}

\ \\
\indent Let us denote by $\sol\sol$ the set of critical
subsolutions for $H$, i.e. the subsolutions of equation \eqref{eq
hjc}. We define the set
 \begin{equation}\label{def D}
\DD:=\bigcap_{v\in\sol\sol}\,\left\{y\in M\,:\,\hbox{$v$ and $S_y$
are differentiable at $y$, \ $Dv(y)=DS_y(y)$\,} \right\},
\end{equation}
where $S_y$ stands for the function $S(y,\cdot)$. The following
holds:

\begin{prop}\label{prop D}
Let $H$ be a strictly convex Hamiltonian. Then $\DD$ is a dense
subset of $\A$. When $M$ is compact, we have in particular that
$\DD$ is a uniqueness set for the critical equation, i.e. if two
critical solutions agree on $\DD$, then they agree on the whole
$M$.
\end{prop}

\begin{dimo}
It is clear by Proposition \ref{prop FatSic} that $\DD$ is
contained in $\A$. Pick $y\in\A$ and choose a static curve
$\gamma\in\K$ passing through $y$. According to Theorem \ref{teo
diff}, there exists a sequence of points $y_n\in\gamma(\R)\cap\DD$
converging to $y$. This proves that $\DD$ is dense in $\A$.

The fact that  $\DD$ is a uniqueness set is now a direct
consequence of the fact that $\A$ is a uniqueness set, see
\cite{FaSic03}.
\end{dimo}

\begin{oss}\label{oss D}
We underline for later use that, by definition of $\DD$, any two
critical subsolutions $u$ and $v$ are differentiable on $\DD$ and
have same gradient.\\
\end{oss}
\end{section}

\begin{section}{Commuting Hamiltonians and critical equations}\label{sez commuting}

The purpose of this section is to explore the relation between the critical equations associated with a pair of commuting Hamiltonians. We open by making precise what we mean by {\em commuting} when referred to a pair of convex Hamiltonians that are just continuous. After deriving a result that will be needed later, we restrict to the case when $M$ is compact and we
look into the corresponding critical equations. We discover in the end that the commutation property
entails very strong informations.\smallskip

Throughout this section $H$ and $G$ will denote a pair of
Hamiltonians satisfying assumptions (H1), (H2) and (H3). The
following notations will be assumed
\begin{itemize}
    \item $L_H$ and $L_G$ are the Lagrangians associated through
    the Fenchel transform with $H$ and $G$,
    respectively.\smallskip
    \item $\S_H$ and $\S_G$ denote the Lax--Oleinik semigroups
    associated with $H$ and $G$, respectively.\smallskip
    \item $h_H^t$ and $h_G^t$ will denote, for every $t>0$, the
    functions associated via \eqref{def h^t} with $H$ and $G$,
    respectively.\smallskip
    \item $h_H$ and $h_G$ are the Peierls barriers associated with
    $H$ and $G$, respectively.\smallskip
\end{itemize}
\begin{definition}\label{def commutation}
We will say that two convex Hamiltonians $H$ and $G$  {\em
commute} if
\begin{equation} \label{commute}
\S_G(s)\big(\S_H(t)\,u\big)(x)=\S_H(t)\big(\S_G(s)\,u\big)(x)\qquad\hbox{for
every $s,\,t>0$ and $x\in M $,}
\end{equation}
and for every function admissible initial datum $u: M
\to\R\cup\{+\infty\}$.
\end{definition}

\begin{oss}
Note that a Hamiltonian function $H$ always commutes with itself.  Also note that, when $M$ is
compact, any continuous function is an admissible initial datum.
\end{oss}

We emphasize that the notion of commutation given in Definition \ref{def commutation} is nothing but a rephrasing of the fact that the the multi--time Hamilton--Jacobi
equation \eqref{multi-time HJ} admits a solution for every Lipschitz continuous initial datum.

A very natural question is that of finding direct and easy--to--check conditions on the Hamiltonians that ensure the commutation property. As explained in the introduction, the problem has been already considered in literature. Here we recall one of the main results proved in \cite{BT}, that can be stated in our setting as follows:

\begin{teorema}\label{BT1}
Let $H$ and $G$ be a pair of convex Hamiltonians, locally Lipschitz in $x$, such that  
\begin{equation*}
\{G,H\}:=\langle{D}_x G,\,{D}_p
H\rangle-\langle{D}_x H,\,{D}_p G\rangle=0\quad\hbox{for a.e. $(x,p)\in M\times\R^N$}.
\end{equation*}
If either $H$ or $G$ is of class $C^1$ on $M\times\R^N$, then \eqref{commute} holds for any $u\in \D{Lip}(M)$.\smallskip
\end{teorema}
\begin{oss}
The definition of commutation given above via \eqref{commute} is
actually equivalent to the cancellation of the Poisson bracket, $\{ \cdot , \cdot \}$,
when the Hamiltonians are additionally assumed of class $C^1$. The
proof of this fact is sketched in the introduction of \cite{BT},
and is detailed in Appendix \ref{appendix commutation} in the case
of Tonelli Hamiltonians. This equivalence will be used to
establish Theorem \ref{teo common strict}, see the proof of Lemma
\ref{lemma common strict}.
\end{oss}

It would be interesting to understand if the null Poisson bracket condition can be somehow relaxed to 
less regular Hamiltonians. For instance, one may wonder if the commutation condition \eqref{commute} holds for pairs of 
locally Lipschitz Hamiltonians having Poisson bracket almost everywhere zero. We will describe in Remark \ref{oss examples} how a non--trivial class of locally Lipschitz Hamiltonians enjoying this property can be provided.

We prove a result that will be needed in the sequel.

\begin{prop}\label{prop sup commutation}
Assume $H_1$ and $H_2$ are two commuting convex Hamiltonians and
set
\[
G(x,p)=\max\{H_1(x,p),\,H_2(x,p)\}\qquad\hbox{for every $(x,p)\in
M \times\R^N$}.
\]
Then $G$ commutes both with $H_1$ and $H_2$.
\end{prop}

\vspace{1ex}

We need three auxiliary results first.

\begin{prop}\label{prop equivalent}
A pair of continuous convex Hamiltonians $H$ and $G$ commute if and only if
\begin{equation}\label{eq comm}
\min_{z\in M }\big(h^t_H(y,z)+h^s_G(z,x)\big) = \min_{z\in M
}\big(h^s_G(y,z)+h^t_H(z,x)\big)
\end{equation}
{for every $x,y\in M $ and $t,\,s>0$.}
\end{prop}

\begin{oss}\label{oss equivalent}
Formula \eqref{eq comm} holds with minima even when $M$ is non
compact. Indeed,
\[
    \tau\,\alpha_*\Big(\frac{|z-\zeta|}{\tau}\Big)
    \leqslant
    h_H^\tau(z,\zeta)
    \leqslant
    \tau\,\beta_*\Big(\frac{|z-\zeta|}{\tau}\Big),
\]
and the same is valid for $h_G^\tau(z,\zeta)$. This readily implies that the infima in \eqref{eq comm} are finite and that every minimizing sequence must stay in a compact subset of $M$. 
\end{oss}
\begin{dimo}
Let $u: M \to\R\cup\{+\infty\}$ be an admissible initial datum.
Using the definitions and the commutation of two nested infima we
get, for every $x\in M $ and $t,s>0$
\begin{eqnarray}
\S_G(s)\big(\S_H(t)\,u\big)(x)
    &=&
    \inf_{z\in M }\,\inf_{\zeta\in M }\Big(
    h_H^t(\zeta,z)+h_G^s(z,x)+u(\zeta)\Big)\nonumber\\
    &=&
    \inf_{\zeta\in M }\Big(\,\inf_{z\in M }\,\big(
    h_H^t(\zeta,z)+h_G^s(z,x)\big)\,+u(\zeta)\Big), \label{inf1} \\
    & & \nonumber\\
\S_H(t)\big(\S_G(s)\,u\big)(x)
    &=&
    \inf_{z\in M }\,\inf_{\zeta\in M }\Big(
    h_G^s(\zeta,z)+h_H^t(z,x)+u(\zeta)\Big)\nonumber\\
    &=&
    \inf_{\zeta\in M }\Big(\,\inf_{z\in M }\,\big(
    h_G^s(\zeta,z)+h_H^t(z,x)\big)\,+u(\zeta)\Big). \label{inf2}
\end{eqnarray}
Now, if $H$ and $G$ commute, then \eqref{eq comm} follows by
plugging in the above equalities as $u$  the function equal to 0
at $y$ and $+\infty$ elsewhere, for every fixed $y\in M $.
Conversely, if \eqref{eq comm} holds true, then \eqref{inf1} and
\eqref{inf2} are equal for any admissible $u$, so $H$ and $G$
commute.\medskip
\end{dimo}

We set $L(x,q):=\min\{L_{H_1}(x,q),\,L_{H_2}(x,q)\}$ for all
$(x,q)\in M \times\R^N$. To ease notations, in the sequel we will
write $L_i$, $h^t_i$ in place of $L_{H_i}$, $h^t_{H_i}$. We recall
that $L^*$ denotes the Fenchel transform of $L$, defined according
to \eqref{def L}.

\begin{lemma}\label{lemma h_G}
For every $x,\,y\in M $ and $t>0$
\begin{equation}\label{eq h^t_G}
h^t_G(x,y)=\inf\left\{\int_0^t L(\gamma,\dot\gamma)\,\dd
s\,:\,\gamma\in C^{1}([0,t]; M
),\,\gamma(0)=x,\,\gamma(t)=y\,\right\}.
\end{equation}
\end{lemma}

\begin{dimo}
By classical results of Calculus of Variations, see for instance
\cite{BuGiHi98}, we  know that the infimum appearing in
\eqref{eq h^t_G} agrees with
\[
\inf\left\{\int_0^t L^{**}(\gamma,\dot\gamma)\,\dd s\,:\,\gamma\in
W^{1,1}([0,t]; M ),\,\gamma(0)=x,\,\gamma(t)=y\,\right\},
\]
so to conclude we only need to prove that $L^{**}=L_G$. From the
inequalities $G\geqslant H_i$ we derive $L_G\leqslant L_i$ for
$i\in\{1,\,2\}$, so $L_G\leqslant L$. By duality
\[
G=L_G^*\geqslant L^*\geqslant
L_i^*=H^i,\qquad\hbox{$i\in\{1,\,2\}$,}
\]
so $G\geqslant L^*\geqslant\max\{H_1,\,H_2\}=G$. Hence $G=L^*$ and
consequently $L_G=L^{**}$.
\end{dimo}

\vspace{2ex}

\begin{lemma}\label{lemma key}
For every $x,\,y\in M $ and $t>0$
\begin{equation*}
h_G^t(x,y)=\inf\Big\{\sum_{i=1}^n
h^{t_i}_{\sigma(i)}(x_{i-1},x_i)\,:\,x_0=x,\,x_n=y,\,\sum_i
t_i=t,\, \sigma\in\{1,\,2\}^n,\,n\in\N\,\Big\}.
\end{equation*}
\end{lemma}

\begin{dimo}
The fact that the right--hand side term of the above equality  is
non smaller than $h_G^t(x,y)$ is an immediate consequence of the
inequalities $L_i\geqslant L_G$ for $i\in\{1,\,2\}$. To prove the
opposite inequality, in view of Lemma \ref{lemma h_G}, it suffices
to show that for every $\eps>0$ and for every curve
$\gamma:[0,t]\to M $ of class $C^1$ joining $x$ to $y$ we have
\begin{equation*}
    \eps+\int_0^t L(\gamma,\dot\gamma)\,\dd s
    \geqslant
    \sum_{i=1}^n \int_{t_{i-1}}^{t_i}
    L_{\sigma(i)}(\gamma,\dot\gamma)\,\dd s
\end{equation*}
for a suitable choice of $n\in\N$, $\{t_i\,:\,0\leqslant i
\leqslant n\,\}$ and $\sigma\in\{1, 2\}^n$.

To this purpose, choose a sufficiently large positive number $R$
such that $\|\dot\gamma\|_\infty<R$ and $\gamma([0,t])\subseteq
B_R$. Denote by $\omega$ a continuity modulus for $L_1$ and $L_2$
in $ B_R \times B_R$. Let $r$ be an arbitrarily chosen positive
number and choose $n\in\N$ large enough in such a way that
\[
\displaystyle{|\gamma(s)-\gamma(\tau)|+|\dot\gamma(s)-\dot\gamma(\tau)|<r\qquad\hbox{for
any $s,\,\tau\in [0, t]$ with $|s-\tau|<\frac{t}{n}$}.}
\]
Let $t_i:=i\,t/n$ for $0\leqslant i \leqslant n$ and define
$\sigma\in\{1,2\}^n$ in such a way that
\[
L_{\sigma(i)}\big(\gamma(t_i),\dot\gamma(t_i)\big)=L\big(\gamma(t_i),\dot\gamma(t_i)\big)
\qquad \hbox{for every $1\leqslant i \leqslant n$.}
\]
For every $s\in [t_{i-1},\,t_i]$ we get
\begin{eqnarray*}
    \left|
     L_{\sigma(i)}\big(\gamma(s),\dot\gamma(s)\big)-L\big(\gamma(s),\dot\gamma(s)\big)
    \right|
    \leqslant
    \left|
    L_{\sigma(i)}\big(\gamma(s),\dot\gamma(s)\big)- L_{\sigma(i)}\big(\gamma(t_i),\dot\gamma(t_i)\big)
    \right|\\
    +
    \left|
     L\big(\gamma(t_i),\dot\gamma(t_i)\big)- L\big(\gamma(s),\dot\gamma(s)\big)
    \right|
    \leqslant
    2\,\omega(r).
\end{eqnarray*}
Then
\[
    \sum_{i=1}^n \int_{t_{i-1}}^{t_i}
    L_{\sigma(i)}(\gamma,\dot\gamma)\,\dd s
    \leqslant
    \int_0^t L(\gamma,\dot\gamma)\,\dd s +2t\,\omega(r),
\]
and the assertion follows by choosing $r$ small enough.
\end{dimo}
\ \\

\noindent{\bf Proof of Proposition \ref{prop sup commutation}.}
Let us prove that $G$ commutes with $H_i$, where $i$ has been
fixed, say $i=1$ for definitiveness. In view of Proposition
\ref{prop equivalent}, we need to show that \eqref{eq comm} holds
with $H_1$ and $G$ in place of $H$ and $G$, respectively. Let us
fix $t,s>0$ and $x,\,y\in M $. Let us show that
\begin{equation}\label{ineq 2}
    \min_{z\in M }\big(h^t_G(x,z)+h^s_{H_1}(z,y)\big)
    \geqslant
    \min_{z\in M }\big(h^s_{H_1}(x,z)+h^t_G(z,y)\big).
\end{equation}
In view of Lemma \ref{lemma key}, it will be enough to prove that,
for every $n\in\N$, the following inequality holds:
\begin{equation}\label{ineq 1}
    \sum_{i=1}^n h^{t_i}_{\sigma(i)}(x_{i-1},x_i)+h^s_{H_1}(x_n,y)
    \geqslant
    \min_{z\in M }\big(h^s_{H_1}(x,z)+h^t_G(z,y)\big)
\end{equation}
for every  $\sigma\in\{1,2\}^n$, $\{x_i\,:\,0\leqslant i\leqslant
n\,\}$ with $x_0=x$, $\{t_i\,:\,0\leqslant i\leqslant n\,\}$ with
$\sum_i t_i=t$.

The proof will be by induction on $n$.

\noindent For $n=1$ inequality \eqref{ineq 1} holds true for
\[
    h^{t}_{\sigma(1)}(x,x_1)+h^s_{H_1}(x_1,y)
    \geqslant
    \min_{z\in M }\big(h^t_{\sigma(1)}(x,z)+h^s_{H_1}(z,y)\big),
\]
and we conclude since $H_{\sigma(1)}$ and $H_1$ commute and
$h^t_{\sigma(1)}\geqslant h^t_G$, for every
$\sigma(1)\in\{1,\,2\}$.\\
\indent Let us now assume that \eqref{ineq 1} holds for $n$ and
let us show it holds for $n+1$. Let $\sigma\in\{1,2\}^{n+1}$,
$\{x_i\,:\,0\leqslant i\leqslant n+1\,\}$ with $x_0=x$,
$\{t_i\,:\,0\leqslant i\leqslant n+1\,\}$ with $\sum_i t_i=t$. We
have
\begin{eqnarray*}
    &&\quad\sum_{i=1}^n
    h^{t_i}_{\sigma(i)}(x_{i-1},x_i)
    +
    h^{t_{n+1}}_{\sigma(n+1)}(x_{n+1},y)
    +
    h^s_{H_1}(x_n,y)\\
    &&\geqslant
    \sum_{i=1}^n
    h^{t_i}_{\sigma(i)}(x_{i-1},x_i)
    +
    \min_{z\in M }\Big(h^{t_{n+1}}_{\sigma(n+1)}(x_n,z)+h^s_{H_1}(z,y)\Big)\\
    &&=
    \sum_{i=1}^n h^{t_i}_{\sigma(i)}(x_{i-1},x_i)
    +
    \min_{z\in M }\Big(h^s_{H_1}(x_n,z)+
    h^{t_{n+1}}_{\sigma(n+1)}(z,y)\Big),
\end{eqnarray*}
where we used the fact that $H_1$ and $H_{\sigma(n+1)}$ commute.
Let us denote by $\overline z$ a point realizing the minimum in
the last row of the above expression. By making use of the
inductive hypothesis we get
\begin{eqnarray*}
    &&\quad\sum_{i=1}^n h^{t_i}_{\sigma(i)}(x_{i-1},x_i)
    +
    h^s_{H_1}(x_n,\overline z)+ h^{t_{n+1}}_{\sigma(n+1)}(\overline z,y)\\
    &&\geqslant
    \min_{\zeta\in M }\Big(h^{s}_{H_1}(x,\zeta)+h^{t-t_{n+1}}_{G}(\zeta,\overline z)\Big)
    +
    h^{t_{n+1}}_{\sigma(n+1)}(\overline z,y)\\
    &&=
    \min_{\zeta\in M }\Big(h^{s}_{H_1}(x,\zeta)+h^{t-t_{n+1}}_{G}(\zeta,\overline z)
    +
    h^{t_{n+1}}_{\sigma(n+1)}(\overline z,y)\Big)\\
    &&\geqslant
    \min_{\zeta\in M }\Big(h^{s}_{H_1}(x,\zeta)+h^{t}_G(\zeta,y)\Big).
\end{eqnarray*}
The opposite inequality in \eqref{ineq 2} comes
in an analogous way. The proof is complete.\qed

\begin{oss}\label{oss examples}
By suitably modifying the above arguments, we can prove the following more general version of  
Proposition \ref{prop sup commutation}: let $H_1$ and $H_2$ be a pair of commuting continuous Hamiltonians  and let $f:\R^2\to\R$ be a convex and increasing function. Here increasing means that 
\[
 f(a_1,a_2)\leq f(b_1,b_2)\qquad\hbox{if $a_i\leq b_i$ for $i=1,\,2$.}
\]
Then $H_1$ and $H_2$ commute with $f(H_1,H_2)$. The proof of this fact requires some extra work that is beyond the purpose of the current paper. Here we just want to explain how this procedure can be used to provide new non--trivial examples of commuting continuous Hamiltonians. 
Let $H_1$ be locally Lipschitz and $H_2$ of class $C^1$ such that their Poisson bracket is almost everywhere zero. Then we know from \cite{BT} that $H_1$ and $H_2$ commute. According to what is stated above, $H_1$ and $G:=f(H_1,H_2)$ commute. It does not seem to us that this example can be deduced from existing results in literature.
\end{oss}

We know restrict our attention to the case  $M=\T^N$ and we investigate on the relation between 
the associated critical equations. We
will denote by $c_H$ and $c_G$ the corresponding critical values
of $H$ and $G$, respectively. Up to adding a constant to the
Hamiltonians, we will assume that $c_H=c_G=0$. Note that this does
not affect the commutation property. The symbols $S_H,\,S_G$ and
$\A_H,\,\A_G$ refer to the critical semidistance and the Aubry set
associated with $H$ and $G$, respectively.

We will also denote by $\sol\sol_H$ and $\sol_H$ the set of
subsolutions and solutions of the critical equations $H=0$,
respectively, and by  $\sol\sol_G$ and $\sol_G$ the
analogous objects for the critical equation $G=0$.\smallskip\\
\indent We start with two results which exploit the fact that $H$
and $G$ commute. Actually, the first result is a direct
consequence of the monotonicity of the semigroups and does not
require $M$ to be compact. The second one uses the fact that the
Lax--Oleinik semigroups are weakly contracting for the infinity
norm and the proof is done applying DeMarr's theorem on existence
of common fixed points for commuting weakly contracting maps on
Banach spaces \cite{DeM}. The compactness of $M$ is crucial to
assure that such common fixed points are critical solutions for
both the Hamiltonians.

The proofs of these results may be found in \cite{Za} and will be
omitted.
\begin{prop}\label{stabilite}
Let $H$ and $G$ be a pair of commuting convex Hamiltonians. Then,
for every $t>0$, we have
\begin{eqnarray*}
\S_H(t)u\in\sol\sol_G\qquad&\ &\hbox{for every $u\in\sol\sol_G$,}\\
\S_H(t)u\in\sol_G\!\quad\qquad&\ &\hbox{for every $u\in\sol_G$.}
\end{eqnarray*}
\end{prop}

\vspace{3ex}

\begin{prop}\label{prop common sol}
Let $H$ and $G$ be a pair of commuting convex Hamiltonians. Then
there exists $u_0\in\sol_H\cap\sol_G$. In particular,
\[
H\big(x,Du_0(x)\big)=G\big(x,Du_0(x)\big)=0
\]
at any differentiability point $x$ of $u_0$.
\end{prop}

\vspace{1ex} We now assume strict convexity of the Hamiltonians
and we exploit the differentiability properties of critical
subsolutions established in Section \ref{sez differentiability} to
prove the following

\begin{teorema}\label{teo same sol}
Let $H$ and $G$ be a pair of strictly convex Hamiltonians. If $H$
and $G$ commute, then $\sol_H=\sol_G$.
\end{teorema}

\begin{dimo}
It is enough to show that $\sol_G\subseteq\sol_H$, since the
opposite inclusion follows by interchanging the roles of $H$ and
$G$.

Take $u\in\sol_G$. To prove that $u\in\sol_H$, it suffices to
show, in view of Proposition \ref{prop critical sol}--{\em (ii)},
that
\[
\S_H(t)u=u\quad\hbox{on $\T^N$\qquad for every $t>0$.}
\]
Since $\S_H(t)u\in\sol_G$, according to Proposition \ref{prop D}
it suffices to prove that
\begin{equation*}
\S_H(t)u=u\quad\hbox{on $\DD_G$\qquad for every $t>0$.}
\end{equation*}
Let $u_0\in\sol\sol_H\cap\sol\sol_G$, and pick a point
$y\in\DD_G$. By definition of $\DD_G$, the function $u_0$ is
differentiable at $y$. Moreover, see Remark \ref{oss D}, for every
$v\in\sol\sol_G$
\[
\hbox{$v$ is differentiable at $y$\qquad and \qquad
$Dv(y)=Du_0(y)$}.\smallskip
\]
Then the function
$w(t,x):=\big(\S_H(t)u\big)(x)$ is differentiable at $y$ for every
$t>0$ and
\begin{equation}\label{eq importante}
H\big(y,D_x w(t,y)\big)=H\big(y,Du_0(y)\big)=0\qquad\hbox{for every $t>0$}
\end{equation}
by Proposition \ref{prop common sol}. 

Now we use the fact that $w$ is a solution
of the evolutive equation
\[
{\partial_t w}+H(x,D_x w)=0\qquad\hbox{in
$(0,+\infty)\times\T^N$.}
\]
The underlining idea is very simple. To focus this point, we will first establish the result by adding a mild regularity assumption on $w$. Then we will deal with the general case, where some technicalities arise.\medskip\\
{\bf First Case: $w$ is locally semiconcave in $(0,+\infty)\times\ \T^N$.}\smallskip\\
\indent This condition is always fulfilled if, for instance, $H$ is locally Lipschitz continuous in $x$, see \cite{CanSon}. Since the map $t\mapsto w(t,y)$ is Lipschitz continuous, it is differentiable for a.e. $t>0$. In view of Lemma \ref{lemma semiconcave} and of \eqref{eq importante}, we infer 
\[
{\partial_t w}(t,y)={\partial_t w}(t,y)+H\big(y,D_x w(t,y)\big)=0\qquad\hbox{for a.e. $t>0$,}
\]
yielding that $w(\cdot,y)$ is constant in $\R_+$. Hence
\[
\big(\S_H(t)u\big)(y)=u(y)\qquad\hbox{for every $t>0$,}
\]
as it was to be shown.\medskip\\
{\bf The general case.}\smallskip\\ 
\indent We only need to prove that $w(\cdot,y)$ is constant, i.e. that 
\[
{\partial_t w}(t,y)=0\qquad\hbox{for a.e. $t>0$.}
\]
First we recall that, by convexity of the Hamiltonian, the fact that $w$ is a subsolution of the evolutive equation is  equivalent to requiring 
\[
p_t+H(x,p_x)\leq 0\qquad\hbox{for every $(p_t,p_x)\in\partial_c w(t,x)$}
\]
for every $(t,x)\in (0,+\infty)\times\T^N$. 
Now the functions $\{ w(\cdot,x)\,:\,x\in\T^N\}$ are locally equi--semiconcave in $(0,+\infty)$, see Lemma \ref{lemma Maxime}. Moreover $w$ has partial derivatives at $(t,y)$ for a.e. $t>0$, so in view of Lemma \ref{lemma t-semiconcave} we get
\begin{equation}\label{ineq1 pain}
{\partial_t w}(t,y)={\partial_t w}(t,y)+H\big(y,D_x w(t,y)\big)\leq 0\qquad\hbox{for a.e. $t>0$.}
\end{equation}
Let us prove the opposite inequality, i.e. 
\[
\partial_t w(t,y)\geq 0\quad\hbox{at any differentiability point $t>0$ of the function $w(\cdot,y)$.} 
\]
In fact, if this were not the case, there would exist $t_0>0$ such that 
$w(\cdot,y)$ is differentiable at $t_0$ and 
\[
 \partial_t w(t_0,y)<-\eps\qquad\hbox{for some $\eps>0$.}
\]
Since $w$ is locally semiconcave in $t$, uniformly with respect to $x$, we infer that there exist $r>0$ such that
\begin{equation}
 \partial_t w(t,x)<-\eps\qquad\hbox{for a.e. $(t,x)\in B_{r}(t_0)\times B_{r}(y)$.}
\end{equation}
This follows from \cite[Theorem 3.3.3]{CaSi00}, which implies here the continuity of $\partial_t w$ with respect to $(t,x)$ on its domain of definition (via an argument analogous to the one used in the proof of Lemma \ref{lemma semiconcave}).

By Lemma \ref{lemma pain}, we infer that
\[
 H\Big(x,D\big( \S_H(t_0)u\big)\Big)\geq \eps\qquad\hbox{in $B_r(y)$}
\]
in the viscosity sense. 
On the other hand, $\S_H(t_0)u\in\sol_G$, it is hence differentiable at $y$ and 
\[
 H\Big(y,D\big(\S_H(t_0)u\big)(y)\Big)=0,
\]
yielding a contradiction. 
\end{dimo}

\medskip
Theorem \ref{teo same sol} has very strong consequences from the
weak KAM theoretic viewpoint. Indeed, we have

\begin{teorema}\label{aubryeq}
Let $H$ and $G$ be a pair of commuting, strictly convex
Hamiltonians. Then
\begin{itemize}
    \item[{\em (i)}]\quad $h_H= h_G$ on $\T^N\times\T^N$;\smallskip
    \item[{\em (ii)}]\quad $\A_H=\A_G$;\smallskip
    \item[{\em (iii)}]\quad $S_H(x,y)=S_G(x,y)$ if either $x$ or
    $y$ belong to $\A_H=\A_G$.
\end{itemize}
\end{teorema}

\begin{dimo}\quad
{\em (i)} Let us arbitrarily fix $y\in\T^N$. By Proposition
\ref{prop h}, $h_H(y,\cdot)$ and $h_G(y,\cdot)$ both belong to
$\sol_H=\sol_G$, so
\[
\S_G(s)\,h_H(y,\cdot)=h_H(y,\cdot),\qquad\qquad \S_H(t)\,h_G(y,\cdot)=h_G(y,\cdot)
\]
for every $s,t>0$. Moreover
\[
h_H^t\ \underset{t\to +\infty}\ucv \ h_H\qquad\hbox{and}\qquad
h_G^s\ \underset{s\to +\infty}\ucv \ h_G\qquad\qquad\hbox{in
$\T^N\times \T^N$.}
\]
Let us denote by $u$ the function equal to $0$ at $y$ and
$+\infty$ elsewhere. For every $s>0$ we have
\begin{eqnarray*}
    &&\hspace{-5ex}h_H(y,\cdot)
    =
    \S_G(s)\,h_H(y,\cdot)
    =
    \lim_{t\to +\infty} \S_G(s)\,h^t_H(y,\cdot)\\
    &&=
    \lim_{t\to +\infty} \S_G(s)\,\S_H(t)u
    =
    \lim_{t\to +\infty} \S_H(t)\,\S_G(s)u
    =
    \lim_{t\to +\infty} \S_H(t)\,h^s_G(y,\cdot).
\end{eqnarray*}
We derive
\begin{eqnarray*}
&&\|h_H(y,\cdot)-h_G(y,\cdot)\|_\infty
    =
    \lim_{t\to +\infty}
    \left\|\S_H(t)\big(\,h^s_G(y,\cdot)\big)-\S_H(t)\big(h_G(y,\cdot)\big)\right\|_\infty\qquad\qquad\qquad\qquad\\
   &&\qquad\qquad\qquad\qquad\qquad\qquad\qquad\qquad\qquad\qquad\qquad\quad  \leqslant
    \|h^s_G(y,\cdot)-h_G(y,\cdot)\|_\infty,
\end{eqnarray*}
and the assertion follows sending $s\to +\infty$.

Assertions {\em (ii)} and {\em (iii)} are a direct consequence of
{\em (i)} in view of Theorem \ref{teo A} and of Proposition
\ref{prop h}, respectively.
\end{dimo}

\vspace{3ex} Next, we show that $H$ and $G$ admit a common strict
subsolution.

\begin{teorema}\label{teo common strict}
Let $H$ and $G$ be a pair of commuting, strictly convex
Hamiltonians, and let $\A$ denote $\A_H=\A_G$. Then there exists
$v\in\sol\sol_H\cap\sol\sol_G$ which is $C^\infty$ and strict in
$\T^N\setminus\A$  both for $H$ and for $G$, i.e.
\begin{equation}\label{eq common strict}
H\big(x,Dv(x)\big)<0\quad\hbox{and}\quad G\big(x,Dv(x)\big)<0\qquad\hbox{for every
$x\in\T^N\setminus\A$.}
\end{equation}
If   $H$ and  $G$ are locally Lipschitz continuous in
$\T^N\times\R^N$, then $v$ can be additionally chosen in
$C^1(\T^N)$.

Finally, if $H$ and $G$ are Tonelli, then $v$ can be chosen in
$C^{1,1}(\T^N)$.
\end{teorema}

\noindent{\bf Proof.} Let us set
\[
F(x,p):=\max\{H(x,p),\,G(x,p)\}\qquad\hbox{for every
$(x,p)\in\T^N\times\R^N$.}
\]
This new Hamiltonian still satisfies (H1), (H2)$'$ and (H3).
Moreover any $u\in\sol_H=\sol_G$ solves the equation
\[
F(x,Du)=0\qquad\hbox{in $\T^N$}
\]
in the viscosity sense, as it is easily seen by definition of $F$.
This yields $c_F=0$ and, according to Proposition \ref{prop sup
commutation} and Theorem \ref{aubryeq}, $\A_F=\A$.

We now invoke the results proved \cite{FaSic03}: by Theorem 6.2,
there exists a critical subsolution $v$ for $F$ which is strict
and smooth in $\T^N\setminus\A$. If $H$ and $G$ are locally
Lipschitz, the same holds for $F$, so $v$ can be additionally
chosen of class $C^1$ on the whole $\T^N$ in view of Theorem 8.1.
The inequalities \eqref{eq common strict} follow since $F\geqslant
H,\,G$.

If now $H$ and $G$ are Tonelli Hamiltonians,  the commutation
property is equivalent to the fact that the Poisson bracket
$\{H,G\}=0$ everywhere, as explained in Appendix \ref{appendix
commutation}. Starting with a $C^1$ (or in fact any) common strict
subsolution $v$, it is possible to realize, as in \cite{Be}, a
Lasry--Lions regularization $v_0$ of $v$, using alternatively the
positive and negative semigroups of $H$. More precisely,
$$v_0=\S_H(t)\left(\S^+_H(s)\, v \right)$$
for $s$ and $t$ suitably chosen, where the positive Lax--Oleinik
semigroup is defined as follows:
\[
\S^+_H(s)v=-\big(\S_{\check H}(-v)\big).
\]
%
Note that $v_0$ is still a subsolution both for $G$ and for $H$, see Remark \ref{oss subsol} and Proposition \ref{stabilite}. The fact that it is strict in
$\mathbb{T}^N\setminus \A$ is proved in the next lemma. The fact
that $v_0$ is $C^{1,1}$ for $t$ and $s$ small enough is proved in
\cite{Be}. \qed
\medskip

We recall that  a Lipschitz subsolution $v\in \sol\sol_G$ is said
to be {\em strict} in an open set $U\subset \T^N$ if for any
$x_0\in U$ there is a neighborhood $V$ of $x_0$ and a constant
$\varepsilon>0$ such that $G\big(x,D v(x)\big)< -\varepsilon$ almost
everywhere in $V$.

Note that if $v$ is $C^1$, it is strict on $U$ if and only if $G\big(x,D v(x)\big)<0$ for any $x\in U$.

\begin{lemma}\label{lemma common strict}
Let $G$ and $H$ be two commuting Tonelli Hamiltonians. Assume $v$
is a critical subsolution for $G$ which is strict outside $\A$.
Then, for all $t>0$, both $\S_G(t)\,v$ and $\S_H(t)\, v$ are
critical subsolutions for $G$, strict outside $\A$.
\end{lemma}

\noindent{\bf Proof.} We will only prove the result for $\S_H(t)\,
v$. The result for $\S_G(t)\, v$ is then a consequence for $G=H$.
We already know by Proposition \ref{stabilite} that $\S_H(t)\, v$
is a critical subsolution of $G$. It is only left to prove the
strict part. This is done in two steps: in a first one, we prove a point wise strictness at differentiability points of $\S_H(t)\, v$. In a second one, we extend this result using Clarke's gradient to any point before concluding.

Let $x\in \T^N\setminus \A$. Consider a curve $\gamma$ verifying
that $\gamma(0)=x$ and
$$\big(\S_H(t)\, v\big)(x)=v\big(\gamma(-t)\big)+\int_{-t}^0 L_H\big(\gamma(s),\dot\gamma (s)\big)\dd s.$$
The curve $(\gamma, \dot \gamma)$ is then a piece of trajectory of
the Euler--Lagrange flow of $H$. It is also known (see
\cite{Fathi} or Proposition \ref{prop Fathi+}) that $D_q L_H
\big(\gamma(-t),\dot\gamma(-t)\big)\in D^- v\big(\gamma(-t)\big)$ and
$$ D_q L_H
\big(\gamma(s),\dot\gamma (s)\big)\in D^+\big(\S_H(t+s)\, v\big)\big(\gamma(s)\big) \qquad\hbox{for every $s\in (-t,0]$.}$$
 Moreover, the
curve $\gamma$ does not intersect $\A$. Indeed, if this were not
the case, the curve $(\gamma, \dot \gamma)$ would be included in
the lifted Aubry set, which is invariant by the Euler--Lagrange
flow of $H$, see \cite{Fathi}, while $x=\gamma(0)\notin \A$. We
therefore deduce that $\gamma(-t)\notin \A$ and, since $v$ is
strict,
$$G\Big(\gamma(-t),D_q L_H \big(\gamma(-t),\dot\gamma(-t)\big)\Big)<0.$$
Now $G$ and $H$ commute; since they are Tonelli, this means their
Poisson bracket is null, see Proposition \ref{justification}.
Otherwise stated, $G$ is constant on the integral curves of the
Hamiltonian flow of $H$, in particular on $s\mapsto\Big(\gamma(s) ,
D_q L_H \big(\gamma(s),\dot\gamma (s)\big)\Big)$. Thus
\[
G\Big(x,D_q L_H \big(x,\dot\gamma (0)\big)\Big)<0,
\]
from which we infer that $G\left(x, D\big(\S_H(t)\,
v\big)(x)\right)<0$ whenever $\S_H(t)\, v$ is differentiable at
$x$. But this is not sufficient to conclude since the function
$\S_H(t)\, v$ is  Lipschitz continuous in $\T^N$, hence
differentiable almost everywhere only. We will prove the
following:\medskip

\noindent{\bf Claim.} Let $x\notin \A$. Then
\[
G(x,p)<0\qquad\hbox{for every $p\in \partial^* \big( \S_H(t)\, v
\big) (x)$,}
\]
where $\partial^* \big( \S_H(t)\, v \big) (x)$ denotes the set
of {reachable gradients} of $\S_H(t)\, v$ at $x$. 
Note that since $\S_H(t)\, v$ is Lipschitz, this set is compact.

Let $p\in \partial^* \big( \S_H(t)\, v \big) (x)$ and consider $x_n\to x$  a sequence of differentiability points for
$\S_H(t)\, v$   such that $D \big(\S_H(t)\, v\big)(x_n)\to p$. For each $n$,
choose a curve $\gamma_n:[-t,0]\to \T^N$ (which is in fact unique)
such that
\[
\big(\S_H(t)\, v\big)(x_n)=v\big(\gamma_n(-t)\big)+\int^0_{-t}
L_H\big(\gamma_n(s),\dot\gamma_n (s)\big)\,\dd s.
\]
For each $n$, the curve $(\gamma_n,\dot \gamma_n)$ is the (only) trajectory of the Euler--Lagrange flow with initial condition verifying $D_q L_H
\big(x_n,\dot\gamma_n (0)\big)= D\big(\S_H(t)\, v\big)(x_n)$. By continuity of this flow,
they uniformly converge, along with their derivatives, to a curve
$\gamma$. By continuity, we obtain
\[
\big(\S_H(t)\, v\big)(x)=v\big(\gamma(-t)\big)+\int_{-t}^0
L_H\big(\gamma(s),\dot\gamma (s)\big)\dd s.
\]
Moreover, by passing to the limit in the equalities $D_q L_H
\big(x_n,\dot\gamma_n (0)\big)= D\big(\S_H(t)\, v\big)(x_n)$, we obtain
\[
D_q L_H \big(x,\dot\gamma (0)\big)=p.
\]
By arguing as above and by exploiting the fact that $x\notin \A$,
we obtain $G(x,p)<0$. Since $G$ is convex, we infer
\[
G(x,p)<0\qquad\hbox{for every $p\in\partial_c \big(\S_H(t)\, v\big)(x)$},
\]
where $\partial_c\big(\S_H(t)\, v\big)(x)$ denotes the {Clarke
differential} of $\S_H(t)\, v$ at $x$, defined as the convex hull
of $\partial^*\big(\S_H(t)\, v\big)(x)$. We now exploit the fact
that the Clarke differential is upper semi--continuous with
respect to the inclusion and point wise compact, see \cite{Cl}. Let $x_0\notin \A$ and
choose $\varepsilon>0$ in such a way that
\[
G(x_0,p)<-2\eps\qquad\hbox{for every $p\in
\partial_c\big(\S_H(t)\, v\big)(x_0)$.}
\]
Then there exists a neighborhood $V$ of $x_0$ such that
\[
G(x,p)<-{\eps}\qquad\hbox{for every $p\in
\partial\big(\S_H(t)\, v\big)(x)$ and $x\in V$.}
\]
In particular, \ $G\big(x,Dv(x)\big)<-\eps$ \ for almost every $x\in V$.
The proof is complete.\qed

\vspace{3ex}

\end{section}

\begin{appendix}
\begin{section}{}\label{appendix b}

The purpose of this Section is to give a self--contained proof of
Theorem \ref{DT1}. We prove two lemmas first. Recall that we are
assuming that the critical value $c$ is equal to 0.

\begin{lemma}\label{lemma utile}
Let $\gamma:[a,b]\to M$ such that
\begin{equation}\label{eq utile}
S\big(\gamma(b),\gamma(a)\big)+\int_a^b  L(\gamma,\dot\gamma)\,\dd s= 0.
\end{equation}
Then $\gamma$ is a static curve.
\end{lemma}

\begin{dimo}
Let $s,t$ be points of $[a,b]$  with $s<t$. We want to prove that
\begin{equation}\label{eq critical}
-S\big(\gamma(t),\gamma(s)\big)=\int_s^t L(\gamma,\dot\gamma)\,\dd \tau=
S\big(\gamma(s),\gamma(t)\big).
\end{equation}
We set $y:=\gamma(b)$ and observe that equality \eqref{eq utile}
can be equivalently written as
\[
S\big(y,\gamma(b)\big)-S\big(y,\gamma(a)\big)=\int_a^b  L(\gamma,\dot\gamma)\,\dd
s.
\]
Since $S(y,\cdot)$ is a critical subsolution, the following hold:
\[
S\big(y,\gamma(b)\big)-S\big(y,\gamma(t)\big)\leqslant \int_t^b
L(\gamma,\dot\gamma)\,\dd s.
\]
and
\[
S\big(y,\gamma(t)\big)-S\big(y,\gamma(a)\big)\leqslant \int_a^t
L(\gamma,\dot\gamma)\,\dd s.
\]
Both inequalities are in fact equalities (summing them up gives an equality) and we obtain
\[
    -S\big(y,\gamma(t)\big)
    =
    S\big(y,\gamma(b)\big)-S\big(y,\gamma(t)\big)
    =
    \int_t^b L(\gamma,\dot\gamma)\,\dd s
\]
for any $t\in [a,b]$. We infer
\[
    0 = S\big(y,\gamma(t)\big)+\int_t^b  L(\gamma,\dot\gamma)\,\dd \tau
    \geqslant
    S\big(y,\gamma(t)\big)+S\big(\gamma(t),y\big)
    \geqslant
    0,
\]
so
\[
S\big(\gamma(t),y\big)=\int_t^b L(\gamma,\dot\gamma)\,\dd \tau =
-S\big(y,\gamma(t)\big).
\]
In particular for every $a\leqslant s<t\leqslant b$
\[
S\big(\gamma(s),y\big)-S\big(\gamma(t),y\big)=\int_s^t  L(\gamma,\dot\gamma)\,\dd
\tau.
\]
The second equality in \eqref{eq critical} then follows since
\[
   S\big(\gamma(s),y\big)-S\big(\gamma(t),y\big)
    \leqslant
    S\big(\gamma(s),\gamma(t)\big)
    \leqslant
    \int_s^t L(\gamma,\dot\gamma)\,\dd \tau.
\]
Let us now prove the other equality in \eqref{eq critical}. By
making use of what was just proved, we have
\begin{eqnarray*}
\int_s^t L(\gamma,\dot\gamma)\,\dd \tau
    &=&
    S\big(\gamma(s),y\big)-S\big(\gamma(t),y\big)\\
    &=&
    -\Big(S\big(y,\gamma(s)\big)+S\big(\gamma(t),y\big)\Big)
    \leqslant
    -S\big(\gamma(t),\gamma(s)\big),
\end{eqnarray*}
and the assertion follows for
\[
    \int_s^t L(\gamma,\dot\gamma)\,\dd
    \tau+S\big(\gamma(t),\gamma(s)\big)
    \geqslant
    S\big(\gamma(s),\gamma(t)\big)+S\big(\gamma(t),\gamma(s)\big)
    \geqslant
    0.
\]
\end{dimo}


%
\medskip

\begin{lemma}\label{lemma 1}
There exists a real number $R>0$ such that
\[
\bigcup_{x\in M}\{q\in\R^N\,:\, L(x,q)=\sigma(x,q)\,\}\subseteq
B_R.
\]
\end{lemma}

\begin{dimo}
By assumption (H3) there exists a constant $\kappa$ such that
$Z_0(x)\subseteq B_{\kappa}$ for every $x\in M$, so
$\sigma(x,q)\leqslant \kappa|q|$ for every $(x,q)\in M\times\R^N$.
By (L3) and by the superlinear and continuous character of
$\alpha_*$, see Remark \ref{oss H3}, there exists a constant
$\alpha_0>0$ such that
\[
(\kappa+1)|q|-\alpha_0\leqslant\alpha_*(|q|)\leqslant
L(x,q)\qquad\hbox{for every $(x,q)\in M\times\R^N$.}
\]
The assertion follows by choosing $R:=\alpha_0$.
\end{dimo}
\bigskip

\noindent{\bf Proof of Theorem \ref{DT1}.} Fix $y\in\A$ and set
$u(\cdot)=S(y,\cdot)$. The function $w(x,t)=u(x)$ is a solution of
the equation
\begin{equation}\label{eq evolutiva}
\partial _t w(x,t)+H\big(x,{D}_x w(x,t)\big)=0
,
\end{equation}
 hence $\S(t)u=u$ for every
$t>0$. In particular, for each $n\in\N$ there exists a curve
$\gamma_n:[-n,0]\to M$ with $\gamma_n(0)=y$ such that
\[
u(y)=u\big(\gamma_n(-n)\big)+\int_{-n}^0 L(\gamma_n,\dot\gamma_n)\,\dd s.
\]
Now $u(y)=0$ and $u\big(\gamma_n(-n)\big)=S\big(y,\gamma_n(-n)\big)$, so by Lemma
\ref{lemma utile} we derive that $\gamma_n$ is a static curve.
Lemma \ref{lemma 1} guarantees that the curves $\gamma_n$ are
equi--Lipschitz continuous, in particular there exists a Lipschtiz
curve $\gamma:\R_-\to M$ such that, up to subsequences,
\[
\gamma_n\ucv\gamma\quad\hbox{in $\R_-$}\qquad\hbox{and}\qquad
\dot\gamma_n\weakcv\dot\gamma\quad\hbox{in}\quad
L^1_{loc}\big(\R_-;\R^N\big).
\]
By a classical semi--continuity result of the Calculus of
Variations \cite{BuGiHi98} we have
\[
 \liminf_{n\to +\infty}\int_a^b L(\gamma_n,\dot\gamma_n)\,\dd s
 \geqslant
 \int_a^b L(\gamma,\dot\gamma)\,\dd s
\]
for every $a<b\leqslant 0$, yielding in particular that $\gamma$
is static too.

We now consider the Hamiltonian $\check H(x,p)=H(x,-p)$. By
Proposition \ref{prop H check}, we know that the critical value
and the Aubry set of $\check H$ agree with $0$ (i.e. the critical
value of $H$) and $\A$. We can apply the previous argument with
$\check S$ and $\check L$ in place of $S$ and $L$ to obtain a
curve $\xi:\R_-\to M$ which is static for $\check H$. We define a
curve $\eta:\R\to M$ by setting
\begin{eqnarray*}
\eta(s):=
\begin{cases}
\xi(-s) & \hbox{if $s\geqslant 0$}\\
\gamma(s) & \hbox{if $s\leqslant 0$}.
\end{cases}
\end{eqnarray*}
We claim that $\eta$ is the  static curve we were looking for. To
prove this, it will be enough, in view of Lemma \ref{lemma utile}, to
show
\begin{equation}\label{eq claim}
 S\big(\eta(b),\eta(a)\big)
    +
    \int_a^b L(\eta,\dot \eta)\,\dd s
    =
    0
\end{equation}
for any fixed $a<0<b$. Indeed, by noticing that $\check
L(x,q)=L(x,-q)$ and $\check S(x,y)=S(y,x)$, we obtain
\begin{eqnarray*}
    \int_0^b L(\eta,\dot \eta)\,\dd s
    =
    \int_{-b}^0 \check L(\xi,\dot \xi)\,\dd s
    =
    -\check S\big(\xi(0),\xi(-b)\big)=-S\big(\eta(b),\eta(0)\big).
\end{eqnarray*}
Hence
\begin{eqnarray*}
    \int_a^b L(\eta,\dot \eta)\,\dd s
    &=&
    \int_a^0 L(\eta,\dot \eta)\,\dd s
    +
    \int_0^b L(\eta,\dot \eta)\,\dd s\\
    &=&
    -S\big(\eta(b),\eta(0)\big)+S\big(\eta(0),\eta(a)\big)
    \leqslant
    -S\big(\eta(b),\eta(a)\big)
\end{eqnarray*}
and \eqref{eq claim} follows since the opposite inequality is
always true. \qed\\
\end{section}

\section{}\label{appendix technical lemmas}

In this appendix we prove three auxiliary lemmas that are needed in the proof of Proposition \ref{prop common sol}.

\begin{lemma}\label{lemma semiconcave}
Let $w(t,x)$ be a locally semiconcave function in $(0,+\infty)\times \T^N$. Then $w$ has partial derivatives  at a point $(t_0,x_0)$ if and only if it is (strictly) differentiable at that point.
\end{lemma}

\begin{dimo}
Let us assume that $w$ has partial derivative at a point $(t_0,x_0)$. It will be enough to show that the set $\partial^* w(t_0,x_0)$ of reachable gradients of $w$ at $(t_0,x_0)$ reduces to the singleton $\big(\partial_t w(t_0,x_0),D_x w(t_0,x_0)\big)$. 
Indeed, let $(p_t,p_x)\in \partial^* w(t_0,x_0)$ and take a sequence $(t_n,x_n)$ of differentiability points of $w$ converging to $(t_0,x_0)$ such that
\[
\partial_t w(t_n,x_n) =:p_{t_n}\to p_t,\qquad D_x w(t_n,x_n)=:p_{x_n}\to p_x
\]
as $n\to +\infty$. The functions 
\[
 \phi_n(t):=w(t-t_0+t_n,x_n),\qquad \psi_n(x)=w(t_n,x-x_0+x_n)
\]
are locally equi--semiconcave in $(0,+\infty)$ and $\T^N$ and differentiable at the points $t=t_0$ and $x=x_0$, respectively.  Moreover
\[
\phi_n \ucv w(\cdot,x_0) \quad \hbox{in $(0,+\infty)$,} \qquad \psi_n \ucv w(t_0,\cdot) \quad \hbox{in $\T^N$}.
\]
By a well known fact about semiconcave functions, see Theorem 3.3.3 in \cite{CaSi00}, we get
\[
 p_{t_n}=\phi_n'(t_0)\to \partial_t w(t_0,x_0),\qquad p_{x_n}=D_x \psi(x_0)\to D_x w(t_0,x_0),
\]
that is $(p_{t_0},p_{x_0})=\big(\partial_t w(t_0,x_0),D_x w (t_0,x_0)\big)$.  
\end{dimo}

\bigskip
In the subsequent lemma, by $\pi_1$ we will denote the projection of $\R\times\R^N$ onto the first variable, i.e.\quad  $\pi_1(p_t,p_x)=p_t$ \quad for every \ $(p_t,p_x)\in\R\times\R^N$.\\
\begin{lemma}\label{lemma t-semiconcave}
Let $w(t,x)$ be a locally Lipschitz function on $(0,+\infty)\times \T^N$. Let us assume that the family of functions  $\{w(\cdot,x)\,:\,x\in\T^N\}$ are locally equi--semiconcave in $(0,+\infty)$. If $w$ has partial derivatives at a point $(t_0,x_0)$, then
\begin{itemize}
 \item[\em (i)] \quad $\pi_1\big(\partial_c w(t_0,x_0)\big)=\{\partial_t w(t_0,x_0)\}$;\medskip
 \item[\em (ii)] \quad $\big(\partial_t w(t_0,x_0), D_x w(t_0,x_0)\big)\in\partial_c w (t_0,x_0)$.\\
\end{itemize}
\end{lemma}

\begin{dimo}
Assertion {\em (i)} follows arguing as in the proof of Lemma \ref{lemma semiconcave} above and exploiting the semiconcavity of $w$ in $t$. To prove item {\em (ii)} we argue as follows. Assume by contradiction that $D_x w(t_0,x_0)$ does not belong to $\pi_2\big(\partial_c w(t_0,x_0)\big)$. Here $\pi_2$ denotes the projection of $\R\times\R^N$ in the second variable, i.e.\quad  $\pi_1(p_t,p_x)=p_x$ \quad for every \ $(p_t,p_x)\in\R\times\R^N$. The set $\pi_2\big(\partial_c w(t_0,x_0)\big)$ is closed and convex, so by Hahn--Banach Theorem there exist a vector $q$ and a constant $a\in\R$ such that
\[
  \langle D_x w(t_0,x_0),\,q\rangle 
  <
  a
  <
  \langle p,q\rangle
  \qquad
  \hbox{for every $p\in\partial_c w(t_0,x_0)$.} 
\]
By upper semicontinuity of Clarke's generalized gradient, the above inequality keeps holding in a neighborhood 
$V$ of $(t_0,x_0)$, i.e. for every $(t,x)\in V$
\[
  \langle D_x w(t_0,x_0),\,q\rangle 
  <
  a
  <
  \langle p,q\rangle
  \qquad
  \hbox{for every $p\in\partial_c w(t,x)$.} 
\]
For $h\not=0$, let us consider the ratio 
\[
 r(h)=\frac{w(t_0,x_0+hq)-w(t_0,x_0)}{h}.
\]
By the Nonsmooth Mean Value Theorem, see Theorem 2.3.7 in \cite{Cl}, there exists a point $x_h$ on the segment joining $x_0$ to $x_0+hq$ and a vector $(\alpha_h,p_{x_h})\in\partial_c w(t_0,x_h)$ such that
\[
 r(h)=\langle (\alpha_h,p_{x_h}),\, (0,q)\rangle= \langle p_{x_h},\,q\rangle.
\]
For $h$ small enough, we infer that $r(h)>a$. On the other hand
\[
 \lim_{h\to 0} r(h)=\langle D_x w(t_0,x_0) ,\,q\rangle<a, 
\]
yielding a contradiction.
\end{dimo}

\medskip
We conclude this appendix by proving the following 

\begin{lemma}\label{lemma pain}
Let $H:\T^N\times\R^N\to\R$ be a continuous function and $w(t,x)$ a Lipschitz function on $\R_+\times\T^N$ satisfying 
\[
 \partial_t w + H(x,Dw)\geqslant 0\qquad\hbox{in $(0,+\infty)\times\T^N$}
\]
in the viscosity sense. Let us assume that 
\begin{itemize}
 \item[\em (i)] the functions $\{\,w(\cdot,x)\,:\,x\in\T^N\,\}$ are locally equi--semiconcave in $(0,+\infty)$;\smallskip
 \item[\em (ii)] there exist a constant $a\in\R$ and two open sets $I\subseteq (0,+\infty)$ and $U\subseteq\T^N$ such that     
\[
 \partial_t w(t,x)<a\qquad\hbox{for a.e. $t\in I$ and for a.e. $x\in U$.}
\]
\end{itemize}
Then, for every $t_0\in I$, the function $u^{t_0}:=w(t_0,\cdot)$ satisfies 
\begin{equation}\label{claim pain} 
 H(x,Du^{t_0})\geqslant -a\qquad\hbox{in $U$}
\end{equation}
in the viscosity sense.
\end{lemma}

\begin{dimo}
We divide the proof in two steps.\medskip

{\bf Step 1. } Let us additionally assume that, for every $t>0$,  the function 
\[
w(t, \cdot)\qquad\hbox{is locally semiconcave in $\T^N$}.
\]
Let $\Sigma:=\{(t,x)\in (0,\infty)\times\T^N\,:\,\hbox{$w$ is not differentiable at $(t,x)$\,}\}$. Then, for a.e. $t>0$, the set 
\[
 \Sigma^t:=\{x\in\T^N\,:\,(t,x)\in\Sigma\,\}
\]
has $N$--dimensional Lebesgue measure equal to 0, so, for any such $t>0$, 
\[
 \partial_t w(t,x)+H\big(x,D w(t,x)\big)\geqslant 0\qquad\hbox{for a.e. $x\in\T^N$.}
\]
In particular, 
\begin{equation}\label{eq pain1} 
H\big(x,Dw(t,x)\big)\geqslant -\partial_t w(t,x)>-a \qquad \hbox{for a.e. $x\in U$.}
\end{equation}
Set $u^t=w(t,\cdot)$. By semiconcavity, 
the inequality \eqref{eq pain1} means that 
\begin{equation}\label{eq pain2}
H(x,Du^t)\geqslant-a\qquad\hbox{in $U$}
\end{equation}
in the viscosity sense. Here we have used the fact that $Du^t$ is continuous on its domain of definition and that the supersolution test is nonempty only at points where $u^t$ is differentiable.  
If now $t_0$ is any point of $I$, we choose a sequence of points $t_n\in I$  converging to $t_0$ for which \eqref{eq pain2} holds for every $n$. Since $u^{t_n}\ucv u^{t_0}$ in $\T^N$, by stability of the notion of viscosity supersolution we get \eqref{claim pain}.\medskip

{\bf Step 2.}  Let $(t_0,x_0)\in I\times U$. Since the functions $w(\cdot,x)$ are locally equi--semiconcave in $t$, we infer that there exists $r>0$ such that
\begin{equation}
 \partial_t w(t,x)<a\qquad\hbox{for a.e. $(t,x)\in B_{2r}(t_0)\times B_{2r}(x_0)$.}
\end{equation}
For every $n\in\N$, set 
\[
 w_n(t,x)=\min_{y\in\T^N}\{w(t,y)+n\,|x-y|^2\}.
\]
Each $w_n(t,\cdot)$ is semiconcave in $\T^N$ for every fixed $t>0$ and satisfies the following inequality in the viscosity sense
\[
{\partial_t w_n}+H(y,D_x w_n)\geqslant -\delta_n\qquad\hbox{in $(0,+\infty)\times\T^N$,}
\]
where $(\delta_n)_{n\in \N}$ is an infinitesimal sequence, see \cite{CaSi00}. Let us denote by $Y(t,x)$ the set of points $y\in\T^N$ which realize the minimum in the definition of $w_n(t,x)$. If $L$ is the Lipschitz constant of $w$ in $\R_+\times\T^N$, it is well known that 
\[
 \D{dist}\big(Y(t,x),x\big)\leqslant \frac{L}{n}\qquad\hbox{for every $n\in\N$}.
\]
Furthermore
\begin{equation}\label{claim derivative}
\partial_t w_n(t',x')=\partial_t w(t',y)\qquad\hbox{for every $y\in Y(t',x')$}
\end{equation}
at any point $(t',x')$ where $w_n$ has partial derivative with respect to $t$. Indeed, if $\varphi(t)$ is a subtangent to  $w_n(\cdot,x')$ at the point $t'$, the function $\varphi(t)-n|x'-y|^2$ is a subtangent to $w(\cdot,y)$ at the point $t'$, so \eqref{claim derivative} follows by semiconcavity of $w_n$ and $w$ with respect to $t$.

In particular, for $n$ big enough, 
\[
 \partial_t w_n(t,x) <a\qquad\hbox{for a.e. $(t,x)\in B_r(t_0)\times B_r(x_0)$.}
\]
By Step 1 we infer that the functions $u^{t_0}_n=w_n(t_0,\cdot)$ satisfy 
\[
 H(x,Du^{t_0}_n)\geqslant -a-\delta_n\qquad\hbox{in $B_r(x_0)$}
\]
in the viscosity sense. Since\  $u_n^{t_0}\ucv u^{t_0}=w(t_0,\cdot)\ \hbox{in $\T^N$}$, 
we conclude by stability that 
\[
 H(x,Du^{t_0})\geqslant -a\qquad\hbox{in $B_r(x_0)$.}
\]
The assertion follows since $t_0$ and $x_0$ were arbitrarily chosen in $I$ and $U$, respectively, together with the fact that the notion of viscosity supersolution is local.  
\end{dimo}
\ \\

\section{}\label{appendix commutation}

In this Appendix, we discuss the equivalence between the notion of commutation given in Definition \ref{def commutation} and the one given in terms of
cancellation of the Poisson bracket when the Hamiltonians are regular enough. In \cite{BT} it is proved that for two convex
$C^1$--Hamiltonians, $G$ and $H$, having null Poisson bracket,
i.e.
\begin{equation*}
\{G,\,H\}:=\langle D_x G,\, D_p H\rangle -\langle D_x H,\, D_p
G\rangle =0\qquad\hbox{in $M\times\R^N$,}
\end{equation*}
the multi--time Hamilton--Jacobi equation \eqref{multi-time HJ}
admits a (unique) viscosity solution for any Lipschitz initial
datum $u_0$. This amounts to saying that the Lax--Oleinik
semigroups commute in the sense of \eqref{commute}. In \cite{BT},
the question of the reciprocal statement is treated by a heuristic
argument. We feel natural to give a neat proof of this fact, at
least in the case of Tonelli Hamiltonians. For clarity of the
exposition, we will place ourself in the case of $M=\mathbb{T}^N$,
but the results remain true if $M=\mathbb{R}^N$.

\begin{prop}\label{justification}
Let $G$ and $H$ be two Tonelli Hamiltonians on $\mathbb{T}^N\times
\R^N$. Assume that
\begin{equation} \label{commute+}
\S_G(s)\big(\S_H(t)\,u\big)(x)=\S_H(t)\big(\S_G(s)\,u\big)(x)\qquad\hbox{for
every $s,\,t>0$ and $x\in \T^N $,}
\end{equation}
and for every admissible initial datum $u: \T^N
\to\R\cup\{+\infty\}$. Then the following relation is identically
verified:
\begin{equation*}
\langle D_x G,\, D_p H\rangle -\langle D_x H,\, D_p G\rangle
=0\qquad\hbox{in $\T^N\times\R^N$.}
\end{equation*}
\end{prop}

In order to prove this, we will use some results about the
behavior of solutions of the Hamilton--Jacobi equation with smooth
initial datum. We introduce some notations. If $f:\T^N\to \R$ is
differentiable, then $\Gamma(f)\subset \T^N\times \R^N$ will
denote the graph of its differential. We will denote by $\phi_G$
(resp. $\phi_H$) the Hamiltonian flow of $G$ (resp. $H$), that is,
the flow generated by the vectorfield
$$X_G(x,p)=\big(x,p,\,D_p G(x,p),\,-D_x G (x,p)\big),\qquad (x,p)\in  \T^N\times \R^N,$$
\Big(resp. $X_H(x,p)=\big(x,p,\,D_p H(x,p),\,-D_x H
(x,p)\big)$\Big).\smallskip

The following is a reformulation of Lemma 3 in \cite{Be}:

\begin{prop}\label{flow}
For any $C^2$ function $u:\T^N\to \R$, there is an $\varepsilon
>0$ such that for any $s,t<\varepsilon$, the functions
$u_{s,t}=\S_G(s)\big(\S_H(t)\,u\big)$ and
$u^{s,t}=\S_H(t)\big(\S_G(s)\,u\big)$ are $C^2$. Moreover,
$\Gamma(u_{s,t})=\phi_G^s\circ \phi _H^t\, \Gamma (u )$ and
$\Gamma(u^{s,t})=\phi_H^t\circ \phi _G^s\, \Gamma (u )$.\smallskip

Finally, for $t<\varepsilon$ fixed (resp. $s<\varepsilon$ fixed),
the function $(s,x) \mapsto u_{s,t}(x)$ \big(resp. $(t,x) \mapsto
u^{s,t}(x)$\big) is a classical solution to the Hamilton Jacobi
equation
$$\frac{\partial u_{s,t}}{\partial s}+G(x, D_x u_{s,t})=0\qquad\hbox{in $(0,+\infty)\times \T^N$},$$
(resp. \quad $\displaystyle\frac{\partial u^{s,t}}{\partial
t}+H(x, D_x u^{s,t})=0\quad\hbox{in $(0,+\infty)\times \T^N$}$).
\end{prop}

\noindent{\bf Proof of Proposition \ref{justification}.} Let us
use Proposition \ref{flow}, differentiating various times the
Hamilton-Jacobi equation, to compute a Taylor expansion of
$u_{s,t}$ for small times and smooth initial datum:
\begin{multline*}
u_{s,t}(x)=  u_{0,t}(x) -s G\big(x, D_x
u_{0,t}(x)\big)-\frac{s^2}{2}\,\big\langle D_p  G\big(x, D_x
u_{0,t}(x)\big),\,
\frac{\partial }{\partial s} D_x u_{0,t}(x)\,\big\rangle+o(s^2)\\
=u_{0,t}(x) -s G\big(x, D_x u_{0,t}(x)\big)+\frac{s^2}{2}\,\big\langle D_p
G\big(x, D_x u_{0,t}(x)\big),\, D_x G\big(x, D_x
u_{0,t}(x)\big)\big\rangle+o(s^2).
\end{multline*}
Notice  that similarly,
$$u_{0,t}(x)=u(x)-t H\big(x, D u(x)\big)+\frac{t^2}{2}\,\big\langle D_p  H\big(x, D u(x)\big),\, D_x H\big(x, D u(x)\big)\big\rangle+o(t^2)$$
and
$$ D_x u_{0,t} (x)=D u(x)-t\left [D_x H\big(x, D u(x)\big)+ D^2 u(x)\,D_p H \big(x, D u(x)\big)\right ]+o(t).$$
By substitution, we obtain the following identity on $\T^N$:
\begin{multline*}
    u_{t,t}
    =
    u-t H(x, D u)+
    \frac{t^2}{2}\,\big\langle D_p  H(x, D u),\,D_x H(x, D u)\big\rangle\\
    -t\big( G(x, D u)
    -t\,\big\langle D_p G(x, D u),\,D_x H(x, D u)+D^2\, u\,D_p H (x, D u)\big\rangle
    \big)\\
 +\frac{t^2}{2}\,\big\langle D_p  G(x, D u),\, D_x G(x, D
 u)\big\rangle+o(t^2),
\end{multline*}
that is,
\begin{eqnarray*}
    u_{t,t}
    =
    u\ -&t&\,\big( H(x, D u)+G(x,Du)\big)+\, t^2\ \big\langle D_p G(x, D u),\,D^2\, u\,D_p H (x, D u)\big\rangle\\
    +&\displaystyle
    \frac{t^2}{2}&\,\big(\big\langle D_p  H(x, D u),\,D_x H(x, D
    u)\big\rangle
    +
    \big\langle D_p  G(x, D u),\, D_x G(x, D u)\big\rangle
    \big)\\
    +&{t^2}&\,\big\langle D_p G(x, D u),\,D_x H(x, D u)\big\rangle\ +o(t^2),
\end{eqnarray*}
We now make the symmetrical computation for $u^{t,t}$ and we
subtract to get
\begin{equation*}
u_{t,t}-u^{t,t}= t^2\big( \big\langle D_p  G(x, D u),\, D_x H(x, D
u)\big\rangle-\big\langle D_p  H(x, D u),\, D_x G(x, D
u)\big\rangle\big) +o(t^2).
\end{equation*}
The left--hand side term is 0 by the commutation hypothesis, so
the assertion follows by letting $t\to 0$ and by exploiting the
fact that $u$, and hence $Du$, is arbitrary. \qed \\
\end{appendix}

\end{document}